\documentclass{article}

\usepackage[english]{babel}

\usepackage[letterpaper,top=2cm,bottom=2cm,left=3cm,right=3cm,marginparwidth=1.75cm]{geometry}

\usepackage{amsmath}
\usepackage{bm}
\usepackage{amssymb}
\usepackage{graphicx}
\usepackage[colorlinks=true, allcolors=blue]{hyperref}
\usepackage{amsthm}
\usepackage{authblk}
\usepackage{comment}
\usepackage{subcaption}
\usepackage{algorithm}
\usepackage{algpseudocode}
\usepackage{amsmath}

\newtheorem{theorem}{Theorem}[section]

\newtheorem{lemma}[theorem]{Lemma}
\newtheorem{proposition}[theorem]{Proposition}

\newtheorem*{repthm:min_existence}{Theorem \ref{thm:min_existence}}

\theoremstyle{definition}
\newtheorem{definition}[theorem]{Definition}

\theoremstyle{remark}
\newtheorem{remark}[theorem]{Remark}

\usepackage{listings}
\usepackage{xcolor}

\definecolor{codegreen}{rgb}{0, 0.6, 0}  

\lstdefinestyle{minimalcode}{
    language=Python,
    basicstyle=\ttfamily\small,
    commentstyle=\color{codegreen},
    showstringspaces=false,
    breaklines=true,
    tabsize=4,
    frame=lines,
    captionpos=b
}

\usepackage{etoolbox}
\pretocmd{\lstlisting}{\vspace{1em}}{}{}
\apptocmd{\endlstlisting}{\vspace{1em}}{}{}

\title{A primal-dual price-optimization method for computing equilibrium prices in mean-field games models}

\author[1]{Xu Wang}
\author[2]{Samy Wu Fung}
\author[1]{Levon Nurbekyan\thanks{Corresponding author. \href{mailto:lnurbek@emory.edu}{lnurbek@emory.edu}}}
\affil[1]{Department of Mathematics, Emory University, Atlanta, GA, USA}
\affil[2]{Department of Applied Mathematics and Statistics, Colorado School of Mines, Golden, CO, USA}

\begin{document}
\maketitle

\begin{abstract}
We develop a simple yet efficient Lagrangian method for computing equilibrium prices in a mean-field game price-formation model. We prove that equilibrium prices are optimal in terms of a suitable criterion and derive a primal-dual gradient-based algorithm for computing them. One of the highlights of our computational framework is the efficient, simple, and flexible implementation of the algorithm using modern automatic differentiation techniques. Our implementation is modular and admits a seamless extension to high-dimensional settings with more complex dynamics, costs, and equilibrium conditions. Additionally, automatic differentiation enables a versatile algorithm that requires only coding the cost functions of agents. It automatically handles the gradients of the costs, thereby eliminating the need to manually form the adjoint equations.

\vskip\baselineskip

\end{abstract}

\section{Introduction}

We are interested in numerical approximations of equilibrium configurations of a class of mean-field game price formation models. Specifically, we assume there is an economy of a continuum of identical agents that optimize their costs for acquiring and selling a single asset in a market. 
Consequently, agents interact indirectly through price formation in the market, where an equilibrium corresponds to a market clearing.

While there are numerous mean-field game price formation models, we focus on the model introduced in~\cite{gomes2016acc,gomessaude'18}, also further discussed in~\cite{ashrafyan2021duality,ashrafyan22potential,ashrafyan2022variational,ashrafyan2024fullydiscrete,gomes2023machine0,gomes23machine,gomes23random}. 
In this model, single agents solve the optimization problem
\begin{equation}\label{eq:oc_single_general}
	\begin{split}
		&\inf_{\alpha} \int_0^T \left[L(z(t),\alpha(t))+\alpha(t)\omega(t) \right] dt+g(z(T))\\
		\mbox{s.t.}~&\dot{z}(t)=\alpha(t),\quad z(0)=x,
	\end{split}
\end{equation}
where $x \in \mathbb{R}$ is the initial state of the agent, representing the initial level of their asset, $\alpha(t)\in \mathbb{R}$ represents the rate at which the agent acquires ($\alpha(t)>0$) or sells ($\alpha(t)<0$) the asset at time $t\in (0,T)$, and $z(t)\in \mathbb{R}$ represents the state (level of the asset) of the agent at time $t$. Furthermore, $L \colon \mathbb{R} \times \mathbb{R} \to \mathbb{R}$ represents the running cost of trading the asset, and $g\colon \mathbb{R} \to \mathbb{R}$ represents the terminal (final) cost the agent has to pay. Finally, $\omega(t)$ is the price of the asset and the term $\alpha(t) \omega(t)$ is the cost for trading at rate $\alpha(t)$. 

Note that the symmetry among the agents yields that optimal trading rates depend only on the price and current asset level of an agent. Thus, $x$ and $\omega$ are the only parameters in the optimization problem~\eqref{eq:oc_single_general}, and we denote by $\alpha^*_{x,\omega}(t),~t\in (0,T)$ a solution of~\eqref{eq:oc_single_general}.
Next, the model assumes that the initial distribution of agents' states (assets) is $\rho_0 \in \mathcal{P}_{ac}(\mathbb{R})$, where $\mathcal{P}_{ac}(\mathbb{R})$ is the space of all compactly supported probability measures which is absolutely continuous with respect to the Lebesgue measure. 
Hence, the total amount of traded assets at time $t$ is equal to
\begin{equation*}
    \int_{\mathbb{R}} \alpha^*_{x,\omega}(t) \rho_0(x)dx.
\end{equation*}
Assuming that the total supply of the asset at time $t$ is $Q(t)$, the market clearing condition reads as 
\begin{equation}\label{eq:market_clearing}
     \int_{\mathbb{R}} \alpha^*_{x,\omega}(t) \rho_0(x)dx=Q(t),\quad \forall t\in (0,T).
\end{equation}
This condition characterizes equilibrium configurations in the model.

\begin{definition}
    $\omega^*$ is an equilibrium price if the market clearing condition~\eqref{eq:market_clearing} holds for $\omega=\omega^*$, where $\alpha^*_{x,\omega}(t),~t\in (0,T)$ represents optimal trading rates in the presence of price $\omega$.
\end{definition}

Our goal is to develop an algorithm for computing equilibrium prices. The basis of our approach is a variational formulation for equilibrium prices. Let
\begin{equation}\label{eq:phi_omega}
    \begin{split}
		\phi_\omega(x)=&\inf_{\alpha} \int_0^T \left[L(z(t),\alpha(t))+\alpha(t)\omega(t) \right] dt+g(z(T))\\
		\mbox{s.t.}~&\dot{z}(t)=\alpha(t),\quad z(0)=x,
	\end{split}
\end{equation}
be the optimal \textit{individual cost} for an initial asset level $x$ when the price function is $\omega$. Next, consider the functionals
\begin{equation}\label{eq:J}
	J[\omega]=\int_{\mathbb{R}} \phi_\omega(x) \rho_0(x)dx,
\end{equation}
and
\begin{equation}\label{eq:I}
    I[\omega]=\int_0^T Q(t)\omega(t)dt - J[\omega].
\end{equation}
Note that $J[\omega]$ represent the optimal \textit{population cost} in the presence of a price function $\omega$. 
It turns out that equilibrium prices are the minimizers of the optimization problem
\begin{equation}\label{eq:mfg_var}
	\inf_\omega I[\omega]=\inf_\omega \int_0^T Q(t)\omega(t)dt - J[\omega].
\end{equation}

The key result that establishes this variational formulation is the relation between the total amount of traded assets and the rate of change of the optimal population cost with respect to the price function. Mathematically, the relation reads as
\begin{equation}\label{eq:dJ_intro}
\delta J[\omega](t)=\int_{\mathbb{R}} \alpha^*_{x,\omega}(t)\rho_0(x)dx, \quad t\in (0,T),    
\end{equation}
where $\delta J[\omega]$ is the Fr\'{e}chet derivative of $J[\omega]$ with respect to $\omega$.

Equation~\eqref{eq:dJ_intro} immediately yields the Fr\'{e}chet derivative of the functional $I[\omega]$ with respect to $\omega$ and implies that the equilibrium prices correspond to the critical points of the functional $I[\omega]$. Thus, our main theoretical result is Theorem~\ref{thm:min_existence}, which we restate here. Assume that
\begin{equation}\label{eq:L}
    L(z,\alpha)=c_0\frac{\alpha^2}{2}+V(z),
\end{equation}
where $c_0>0$, and $V$ is a twice differentiable function with globally bounded first and second derivatives; that is,
\begin{equation}\label{eq:V_bounds}
    \sup_{z\in \mathbb{R}} |V'(z)|+|V''(z)| <\infty.
\end{equation}
Additionally, assume that $g$ is also twice differentiable with globally bounded first and second derivatives:
\begin{equation}\label{eq:g_bounds}
    \sup_{z\in \mathbb{R}} |g'(z)|+|g''(z)| <\infty.
\end{equation}
Then our main theoretical result is as follows.

\begin{repthm:min_existence}
Assume that~\eqref{eq:L},~\eqref{eq:V_bounds},~\eqref{eq:g_bounds} hold. Then, for all $Q\in L^2(0,T)$, the functional $I$ admits a minimizer. Moreover, for every minimizer $\omega^*$, we have that $\omega^*+c_0 Q \in C^1([0,T])$. 

In particular, for every supply $Q\in L^2(0,T)$ there exists an equilibrium price $\omega^*$, which is as smooth as $Q$, up to $C^1([0,T])$.
\end{repthm:min_existence}

\begin{remark}
    Hypotheses~\eqref{eq:L},~\eqref{eq:V_bounds}, and~\eqref{eq:g_bounds} are the standing assumptions of this work, and all theoretical results and numerical experiments are developed under these conditions. The quadratic dependence on the control simplifies parts of the analysis (e.g., Propositions~\ref{prp:min_uniqueness} and~\ref{prp:min_stability}). 

    Extensions to more general running costs appear feasible using techniques from~\cite{nurbekyan_oc_hilbert}. In such cases, explicit computations (for instance, involving the function $a$ in Proposition~\ref{prp:min_uniqueness}) would be replaced by derivations based on the adjoint (momentum) variable 
    \[
        p = -\partial_\alpha L(z,\dot{z}) - \omega,
    \]
    and the Euler–Lagrange equation would be complemented with the corresponding Hamiltonian system. 

    Finally, since the running cost depends linearly on $\omega$, the arguments in Proposition~\ref{prp:phi_omega} remain unaffected. A detailed investigation of these extensions is left for future work.
\end{remark}

Furthermore, our main computational result is the primal-dual Algorithm~\ref{alg:main}, which is based on this previous theorem. More specifically, we first rewrite~\eqref{eq:mfg_var} as a saddle-point problem
\begin{equation}\label{eq:saddle_cont}
    \begin{split}
        &\inf_\omega \sup_\alpha \mathcal{L}(\omega,\alpha)\\
        =& \inf_\omega \sup_\alpha\int_0^T \omega(t) Q(t)dt-\int_{\mathbb{R}} \left[ \int_0^T \left(L(z_{x,\alpha}(t),\alpha_x(t))+\alpha_x(t)\omega(t)\right)dt+g(z_{x,\alpha}(T)) \rho_0(x)dx\right],
    \end{split}
\end{equation}
where
\begin{equation*}
\dot{z}_{x,\alpha} (t)=\alpha_x(t),\quad t\in (0,T),\quad z_{x,\alpha}(0)=x.  
\end{equation*}
Next, we use a version of the primal-dual hybrid gradient (PDHG)~\cite{champock11,champock16} to solve~\eqref{eq:saddle_cont}. In particular, we replace the proximal ascent steps for $\alpha$ by gradient ascent steps, which yields Algorithm~\ref{alg:main} below. Note that we perform proximal descent steps in $\omega$, as in the original PDHG algorithm, which are optimization steps with quadratic (norm-squared) penalty terms (line 5 in Algorithm~\ref{alg:main}).

\begin{algorithm}[H]
\caption*{\textbf{Algorithm~\ref{alg:main}} Primal-dual algorithm for a mean-field game price formation model}
\begin{algorithmic}[1]
\Require Lagrangian $L$, terminal cost function $g$, initial density $\rho_0$, supply function $Q$; step sizes $\tau_\alpha,\tau_\omega > 0$.
\State Initialize $\alpha^{(0)}$, $\overline{\alpha}^0$, $\omega^{(0)}$.
\For{$k = 0, 1, 2, \dots$}
    \State $\alpha^{(k+1)} \gets \alpha^{(k)} + \tau_{\alpha} \nabla_{\alpha} \mathcal{L}(\omega^{(k)}, \alpha)$ (gradient computed via \textbf{automatic differentiation})
    \State $\overline{\alpha}^{(k+1)} \gets 2\alpha^{(k+1)} - \alpha^{(k)}$
    \State $\omega^{(k+1)} \gets \displaystyle\arg\min_{\omega} \mathcal{L}(\omega, \overline{\alpha}^{(k+1)}) + \dfrac{1}{2\tau_\omega} \int_0^T |\omega(t) - \omega^{(k)}(t)|^2 dt $
\EndFor
\end{algorithmic}
\end{algorithm}

Here, we state the continuous version of Algorithm~\ref{alg:main} for the convenience of the reader and to minimize the introduction of new notation. In the discrete version of the algorithm, we replace integrals and ODEs by numerical integration and finite-difference equations, respectively; a detailed discussion of Algorithm~\ref{alg:main} is in Section~\ref{sec:alg}. 

A critical aspect of our implementation of Algorithm~\ref{alg:main} is that the gradient $\nabla_{\alpha} \mathcal{L}(\omega^{(k)}, \alpha)$ is computed via modern automatic differentiation tools. Thus, for instance, we avoid forming the adjoint equations manually. This approach yields an accurate and flexible algorithmic framework that requires coding only the objective function $\mathcal{L}(\omega,\alpha)$ and eliminates the necessity of manually re-deriving adjoint equations for new problem data such as Lagrangians, terminal costs, and dynamics. Listing~\ref{code:primal_dual} contains the code block for the implementation of the main loop in Algorithm~\ref{alg:main} in PyTorch.

As one can see, we only need to specify that among all input variables of the objective function $\verb|compute_objective()|$, we track only $\verb|alpha|$ when computing its gradient via $\verb|.backward()|$. Hence, for a modified problem data, one only needs to update the objective function $\verb|compute_objective()|$, while the module in Listing~\ref{code:primal_dual} remains unchanged. More details and a complete implementation of Algorithm~\ref{alg:main} are available at~\href{https://github.com/lnurbek/price-mfg-solver}{https://github.com/lnurbek/price-mfg-solver}.

\begin{lstlisting}[style=minimalcode, caption={PyTorch code block for the implementation of Algorithm~\ref{alg:main}}, label={code:primal_dual}]
for _ in range(num_iters):
    # Step 1: Create a fresh alpha tensor with gradient tracking
    alpha = alpha_prev.clone().detach().requires_grad_(True)

    # Step 2: Compute objective and gradient
    loss = compute_objective(alpha, omega, x0, dt, sigma, Q, L, g)
    loss.backward()
    grad_alpha = alpha.grad.detach()

    # Step 3: Primal update
    alpha_new = alpha.detach() + (tau_alpha * M / dt) * grad_alpha

    # Step 4: Extrapolation
    alpha_bar = 2 * alpha_new - alpha_prev

    # Step 5: Dual update
    omega = omega + tau_omega * (alpha_bar.mean(dim=0) - Q)

    # Step 6: Prepare for next iteration
    alpha_prev = alpha_new

return alpha_prev, omega
\end{lstlisting}

The remainder of the paper is organized as follows. In Sections~\ref{subsec:related} and~\ref{subsec:modeling}, we discuss the related work and modeling considerations. Furthermore, in Section~\ref{sec:analysis_single}, we analyze theoretical properties of the single-agent optimal control problem~\eqref{eq:oc_single_general}, which are required for further analysis and the proof of the main Theorem~\ref{thm:min_existence}. Section~\ref{sec:I} contains the analysis of the properties of the functional $I$~\eqref{eq:I} and the proof of Theorem~\ref{thm:min_existence}. Next, in Section~\ref{sec:alg}, we introduce the Algorithm~\ref{alg:main}, and in Section~\ref{sec:experiments} we perform numerical experiments to validate our findings. Section~\ref{sec:conclusion} contains concluding remarks. Appendix~\ref{apdx:analytic} contains the derivations of the analytic solutions for a class of linear quadratic models used for benchmarking our algorithm in Section~\ref{sec:experiments}.

\subsection{Related work}\label{subsec:related}

The price formation model considered in this manuscript is closely related to the ones considered in~\cite{gomes2016acc,gomessaude'18,paola19,bonnans2021schauder,graber2021weak} and developed in subsequent works. In~\cite{bonnans2021schauder,graber2021weak}, such models are referred to as \textit{mean-field games of controls}, since the agents are coupled through the distribution of their controls, $\alpha^*_{x,\omega}(t)$, rather than their states as in the earlier MFG models~\cite{LasryLions06a,LasryLions06b,LasryLions2007,HCM06,HCM07}.

In~\cite{ashrafyan22potential,ashrafyan2022variational}, the authors find a variational formulation for the PDE system characterizing the equilibrium price. The variational formulation derived by the authors relies on the characterization of divergence-free vector fields in two space-dimensions via Poincar\'{e}'s lemma, which does not extend to higher-dimensional settings. In contrast, the variational formulation developed here does not depend on the dimensionality and extends to the higher-dimensional setting.

In~\cite{gomes23random,gomes2023machine0,gomes23machine,alharbi_et_al19price,ashrafyan2024fullydiscrete}, the authors rely on the \textit{dual} optimization problem of~\eqref{eq:mfg_var}, where the trajectories are the optimization variables, and the price variable plays a role of a Lagrange multiplier. Thus, they assume the convexity of the running and terminal costs to ensure the dual problem is convex with respect to the trajectories, and yields a unique solution. Our approach does not require the convexity of the running and terminal costs; indeed, $I[\omega]$ is convex regardless of these assumptions, and our approach is not sensitive to the non-uniqueness of the trajectories.

In~\cite{bonnans2021schauder,graber2021weak,bonnans2023lagrangian,bonnans2023discrete,lavigne2023generalized}, the authors address theoretical and computational aspects of related price formation (mean-field games of controls) models, where the price $\omega$ is an \textit{explicit function} of the total amount of traded assets; that is,
\begin{equation*}
\omega(t)=\psi\left( \int_{\mathbb{R}} \alpha^*_{x,\omega}(t) \rho_0(x)dx\right),\quad t\in (0,T),    
\end{equation*}
where $\psi$ is a given \textit{price function}. Recall that in our case the price is \textit{implicitly} determined by a market clearing condition~\eqref{eq:market_clearing} as in~\cite{gomes2016acc,gomessaude'18} and follow-up works.

The works~\cite{bonnans2021schauder,graber2021weak,bonnans2023lagrangian,bonnans2023discrete,lavigne2023generalized} also tend to focus on the trajectories as primal optimization variables but do not require convexity of the running and terminal costs, thanks to working with the probability measures that support the trajectories.

\subsection{Modeling considerations}\label{subsec:modeling}

In this work, we do not impose constraints on agents' states and controls. However, real-world applications and deployment of price formation models require such constraints. We expect our approach to accommodate state and control constraints through the extension of the results in Section~\ref{sec:analysis_single} to the constrained setting. Indeed, there is a well-developed theory of the optimal control theory with state and control constraints~\cite{cannarsa2018bfc11smoothnessconstrainedsolutionscalculus,bonnans2023lagrangian}, which warrants the qualitative properties of the functional $I[\omega]$ that are essential for our main results. We will address the constrained case in our future work.



\section{Analysis of the single-agent optimization problem}\label{sec:analysis_single}

Assuming~\eqref{eq:L}, the optimization problem~\eqref{eq:phi_omega} reduces to
\begin{equation}\label{eq:phi_omega_1}
    \begin{split}
		\phi_\omega(x)=&\inf_{\alpha \in L^2(0,T)} \int_0^T \left[c_0 \frac{\alpha(t)^2}{2}+\alpha(t)\omega(t)+V(z(t)) \right] dt+g(z(T))\\
		\mbox{s.t.}~&\dot{z}(t)=\alpha(t),\quad z(0)=x\\
  =&\inf_{z \in H^1(0,T)}\left\{ \int_0^T \left[c_0 \frac{\dot{z}(t)^2}{2}+\dot{z}(t)\omega(t) +V(z(t))\right] dt+g(z(T))~\text{s.t.}~z(0)=x\right\}\\
	\end{split}
\end{equation}
Note that we eliminated the control $\alpha$ using the time-derivative of the state; this yields a lighter notation. Additionally, we specified the space for the controls to be $L^2(0,T)$, which imples the state trajectory is in the Sobolev space $H^1(0,T)$.
An elementary inequality
\begin{equation}\label{eq:elementary}
    |z(t')-z(t)|\leq \int_t^{t'} |\dot{z}(s)|ds\leq \sqrt{|t'-t|} ~\|\dot z\|_{2}
\end{equation}
yields that $H^1(0,T)$ functions are H\"{o}lder-$\frac{1}{2}$ continuous and have well-defined values everywhere in $[0,T]$. In particular, we can impose the left endpoint constraint $z(0)=x$ and a terminal cost $g(z(T))$. Finally, $\omega\in L^2(0,T)$ implies that the integral cost in~\eqref{eq:phi_omega_1} is finite for all $z\in H^1(0,T)$.

Our first task is to show that~\eqref{eq:phi_omega_1} admits minimizers for all $x\in \mathbb{R}$. The notation $C(q_1,q_2,\cdots)$ denotes positive constants that depend on quantities $q_1,q_2,\cdots$. For the ease of presentation, we will not specify the dependence of such constants on problem data, such as $c_0,V,g,T$. 

\begin{lemma}\label{lma:existence_mins}
    Let $\omega \in L^2(0,T)$ and $x\in \mathbb{R}$ be arbitrary. Then~\eqref{eq:phi_omega_1} admits a minimizer. Moreover, all minimizers $z^*(t,x)$ satisfy the bounds
    \begin{equation}\label{eq:min_bounds}
    \begin{split}
        \|\dot{z}^*(\cdot,x)\|_{2}\leq & C(\|\omega\|_2),\\
        |z^*(t',x)-z^*(t,x)|\leq & C(\|\omega\|_2)~\sqrt{|t'-t|},\quad \forall t,t'\in [0,T],\\
        |z^*(t,x)|\leq& |x|+ C(\|\omega\|_2),\quad \forall t\in [0,T].
    \end{split}
    \end{equation}
\end{lemma}
\begin{proof} We break the proof into several steps.

\noindent \textit{Claim 1.} For every $z\in H^1(0,T)$ such that $z(0)=x$ we have that
\begin{equation}\label{eq:trejactory_lower_bound}
    \int_0^T \left[c_0 \frac{\dot{z}(t)^2}{2}+\dot{z}(t)\omega(t) +V(z(t))\right] dt+g(z(T))\geq T\cdot V(x)+g(x)+ C\|\dot{z}\|_2^2-C\|\omega\|^2-C
\end{equation}
\begin{proof}[Proof of Claim 1]
Firstly, we have that
\begin{equation*}
    \begin{split}
        \int_0^T V(z(t)) dt +g(z(T))\geq & \int_0^T V(x) dt +g(x)-C\int_0^T |z(t)-x|dt-C|z(T)-x|\\
        \geq & T\cdot V(x)+g(x)- C \|\dot{z}\|_2,
    \end{split}
\end{equation*}
where we used~\eqref{eq:elementary} and $z(0)=x$. Next, we have that
\begin{equation*}
    \begin{split}
        \int_0^T \left[c_0 \frac{\dot{z}(t)^2}{2}+\dot{z}(t)\omega(t) \right]dt\geq \int_0^T \left[c_0 \frac{\dot{z}(t)^2}{4}-\frac{\omega(t)^2}{c_0} \right]dt=\frac{c_0\|\dot{z}\|_2^2}{4}-\frac{\|\omega\|_2^2}{c_0}.
    \end{split}
\end{equation*}
The combination of the previous two inequalities and
\begin{equation*}
    \frac{c_0\|\dot{z}\|_2^2}{4}-C\|\dot{z}\|\geq \frac{c_0\|\dot{z}\|_2^2}{8}-C
\end{equation*}
yields~\eqref{eq:trejactory_lower_bound}.
\end{proof}

\noindent \textit{Claim 2.} Bounds~\eqref{eq:min_bounds} hold for all minimizers.
\begin{proof}[Proof of Claim 2.]
    Let us first prove~\eqref{eq:min_bounds}. Note that~\eqref{eq:elementary} and the triangle inequality imply the second and third inequalities if one proves the first one in~\eqref{eq:min_bounds}. Assuming that $z^*(x,t)$ is a minimizer for~\eqref{eq:phi_omega_1}, we have that
    \begin{equation*}
        \int_0^T \left[c_0 \frac{\dot{z}^*(t,x)^2}{2}+\dot{z}^*(t,x)\omega(t) +V(z^*(t,x))\right] dt+g(z^*(T,x))\leq T\cdot V(x)+g(x),
    \end{equation*}
    where we compared $z^*(t,x)$ with $\tilde{z}(t)\equiv x$. Furthermore, applying~\eqref{eq:trejactory_lower_bound} to $z^*(t,x)$ we obtain
    \begin{equation*}
        T\cdot V(x)+g(x)+ C\|\dot{z}^*(\cdot,x)\|_2^2-C\|\omega\|^2-C\leq T\cdot V(x)+g(x),
    \end{equation*}
    which yields the first inequality in~\eqref{eq:min_bounds}.
\end{proof}

\noindent \textit{Claim 3.} Problem~\eqref{eq:phi_omega_1} admits a minimizer.
\begin{proof}[Proof of Claim 3]
    Firstly,~\eqref{eq:trejactory_lower_bound} yields that
\begin{equation*}
    \int_0^T \left[c_0 \frac{\dot{z}(t)^2}{2}+\dot{z}(t)\omega(t) +V(z(t))\right] dt+g(z(T)) \geq T\cdot V(x)+g(x)-C\|\omega\|_2^2-C,
\end{equation*}
and so the infimum $\phi_\omega(x)$ in~\eqref{eq:phi_omega_1} is finite:
\begin{equation}\label{eq:phi_omega_bounds}
    T\cdot V(x)+g(x)\geq \phi_\omega(x)\geq T\cdot V(x)+g(x)-C\|\omega\|_2^2-C,
\end{equation}
where the upper bound comes from the cost of the constant trajectory $\tilde{z}(t)\equiv x$.

Now assume that $z_n \in H^1(0,T)$ is a minimizing sequence for~\eqref{eq:phi_omega_1}. Then we have that 
\begin{equation*}
    \int_0^T \left[c_0 \frac{\dot{z}_n(t)^2}{2}+\dot{z}_n(t)\omega(t) +V(z_n(t))\right] dt+g(z_n(T)) \leq C,\quad \forall n \geq 1.
\end{equation*}
Invoking~\eqref{eq:trejactory_lower_bound} again, we find that
\begin{equation*}
    T\cdot V(x)+g(x)+C \|\dot{z}_n\|_2^2-C\|\omega\|_2^2-C\leq C \Longrightarrow \|\dot{z}_n\|_2 \leq C(\|\omega\|_2).
\end{equation*}
Furthermore, taking into account~\eqref{eq:elementary}, $z_n(0)=x$, and the triangle inequality, we find that
\begin{equation*}
    |z_n(t)-z_n(t')|\leq \sqrt{|t'-t|}~C(\|\omega\|_2),\quad |z_n(t)|\leq |x|+\sqrt{t}~C(\|\omega\|_2),\quad\forall t,t'\in (0,T).
\end{equation*}
Combining the previous inequalities, we obtain that $\|z_n\|_2\leq C(\|\omega\|_2)$, and so $\{z_n\}$ is
\begin{itemize}
    \item a bounded sequence in $H^1(0,T)$,
    \item a uniformly bounded and equicontinuous sequence in $C([0,T])$.
\end{itemize}
Hence, there exists a $z^*(\cdot,x) \in H^1(0,T)$ such that $z^*(0,x)=x$, and up to a subsequence $\{z_n\}$
\begin{itemize}
    \item converges to $z^*(\cdot,x)$ weakly in $H^1(0,T)$,
    \item converges to $z^*(\cdot,x)$ in $C([0,T])$.
\end{itemize}
Hence, we obtain that
\begin{equation*}
    \begin{split}
        &\liminf_{n \to \infty} \int_0^T \left[c_0 \frac{\dot{z}_n(t)^2}{2}+\dot{z}_n(t)\omega(t) +V(z_n(t))\right] dt+g(z_n(T))\\
        \geq &  \int_0^T \left[c_0 \frac{\dot{z}^*(t,x)^2}{2}+\dot{z}^*(t,x)\omega(t) +V(z^*(t,x))\right] dt+g(z^*(T,x)).
    \end{split}
\end{equation*}
Since $\{z_n\}$ is a minimizing sequence, this last inequality yields that $z^*(\cdot,x)$ is a minimizer for~\eqref{eq:phi_omega_1}.
\end{proof}
The proof of of Lemma~\ref{lma:existence_mins} is now complete.
\end{proof}

Next, we study finer properties of the minimizers of~\eqref{eq:phi_omega_1}.
\begin{proposition}\label{prp:min_uniqueness}
    Let $\omega \in L^2(0,T)$ and $x\in \mathbb{R}$ be arbitrary, and $z^*(\cdot,x)$ be a minimizer of~\eqref{eq:phi_omega_1}. Then the function
    \begin{equation}\label{eq:a}
        a(t)=c_0 z^*(t,x)+\int_0^t \omega(s)ds,\quad t\in [0,T]
    \end{equation}
    is twice continuously differentiable and satisfies the second order ordinary differential equation
    \begin{equation}\label{eq:a_ode}
        \begin{cases}
            \ddot{a}(t)=V'(z^*(t,x)),~t\in (0,T),\\
            a(0)=c_0 x,~\dot{a}(T)=-g'(z^*(T,x)).
        \end{cases}
    \end{equation}
        Moreover, if $\phi_\omega$ is differentiable at $x$ then $a'(0)=-\phi_\omega'(x)$, and $z^*(t,x)$ is the unique minimizer of~\eqref{eq:phi_omega_1}.
\end{proposition}
\begin{proof}
    Let $\eta \in C^\infty([0,T])$ such that $\eta(0)=0$. Then we have that
    \begin{equation*}
        \frac{d}{d\epsilon}\bigg|_{\epsilon=0} \int_0^T \left[c_0 \frac{\dot{z}^*(t,x)^2}{2}+\dot{z}^*(t,x)\omega(t) +V(z^*(t,x)+\epsilon \eta(t))\right] dt+g(z^*(T,x)+\epsilon \eta(T))=0,
    \end{equation*}
    which yields
    \begin{equation*}
        \int_0^T \left[(c_0 \dot{z}^*(t,x)+\omega(t))\dot{\eta}(t) +V'(z^*(t,x))\eta(t)\right] dt+g'(z^*(T,x))\eta(T)=0.
    \end{equation*}
    Rewriting this previous equation in terms of $a$, we obtain
    \begin{equation}\label{eq:foc}
        \int_0^T \left[\dot{a}(t)\dot{\eta}(t) +V'(z^*(t,x))\eta(t)\right] dt+g'(z^*(T,x))\eta(T)=0,
    \end{equation}
    where we a priori have that $\dot{a}\in L^2(0,T)$. Choosing $\eta$ such that $\eta(T)=0$, we arrive at
    \begin{equation*}
        \int_0^T \left[\dot{a}(t)\dot{\eta}(t) +V'(z^*(t,x))\eta(t)\right] dt=0, 
    \end{equation*}
    for all $\eta\in C^\infty([0,T])$ such that $\eta(0)=\eta(T)=0$, which means the weak derivative of $\dot{a}(t)$ is $V'(z^*(t,x))$. But the latter is continuous, so $\dot{a}$ is continuously differentiable, and 
    \begin{equation*}
        \ddot{a}(t)=V'(z^*(t,x)),~t\in (0,T).
    \end{equation*}
    Plugging this back into~\eqref{eq:foc} we find that
    \begin{equation}
    \begin{split}
        0=&\int_0^T \left[\dot{a}(t)\dot{\eta}(t) +V'(z^*(t,x))\eta(t)\right] dt+g'(z^*(T,x))\eta(T)\\
        =&\int_0^T \left[\dot{a}(t)\dot{\eta}(t) +\ddot{a}(t)\eta(t)\right] dt+g'(z^*(T,x))\eta(T)=(\dot{a}(T)+g'(z^*(T,x))) \eta(T)
    \end{split}
    \end{equation}
    for all $\eta\in C^\infty([0,T])$ such that $\eta(0)$. Choosing $\eta$ such that $\eta(T)=1$ we obtain
    \begin{equation*}
        \dot{a}(T)=-g'(z^*(T,x)).
    \end{equation*}
    The equality $a(0)=c_0 x$ is evident.

    Now assume that $\phi_\omega$ is differentiable at $x$, and $x'$ is an arbitrary point. Then we have that $\tilde{z}(t)=z^*(t,x)+x'-x$ is in general suboptimal for~\eqref{eq:phi_omega_1} for the initial state $x'$, and so
    \begin{equation*}
        \begin{split}
            \phi_\omega(x')\leq&  \int_0^T \left[c_0 \frac{\dot{z}^*(t,x)^2}{2}+\dot{z}^*(t,x)\omega(t) +V(z^*(t,x)+x'-x)\right] dt+g(z^*(T,x)+x'-x)\\
            = & \int_0^T \left[c_0 \frac{\dot{z}^*(t,x)^2}{2}+\dot{z}^*(t,x)\omega(t) +V(z^*(t,x))\right] dt+g(z^*(T,x))\\
            &+\int_0^T \left[V(z^*(t,x)+x'-x)-V(z^*(t,x))\right] dt+g(z^*(T,x)+x'-x)-g(z^*(T,x))\\
            \leq & \phi_\omega(x)+\int_0^T V'(z^*(t,x)) (x'-x) dt+g'(z^*(T,x))(x'-x)+C(x'-x)^2\\
            =&\phi_\omega(x)+\left[\int_0^T \ddot{a}(t)dt+g'(z^*(T,x)) \right] (x'-x)+C (x'-x)^2\\
            =&\phi_\omega(x)-\dot{a}(0) (x'-x)+C(x'-x)^2,
        \end{split}
    \end{equation*}
    where we used the fact that $V,g$ have globally bounded second derivatives. Consequently, we have that
    \begin{equation*}
        \begin{split}
            \frac{\phi_\omega(x')-\phi_\omega(x)}{x'-x} \leq -\dot{a}(0)+C(x'-x),\quad x'>x,\\
            \frac{\phi_\omega(x')-\phi_\omega(x)}{x'-x} \geq -\dot{a}(0)+C(x'-x),\quad x'<x,\\
        \end{split}
    \end{equation*}
    and passing to the limit as $x'\to x$, we obtain
    \begin{equation*}
        \phi'_\omega(x)=-\dot{a}(0).
    \end{equation*}
    To prove the uniqueness of $z^*(\cdot,x)$, note that \eqref{eq:a_ode} can be written as
    \begin{equation*}
        \begin{cases}
            \ddot{a}(t)=V'(a(t)-\Omega(t)),~t\in (0,T),\\
            a(0)=c_0 x,~\dot{a}(T)=-g'(a(T)-\Omega(T)),
        \end{cases}
    \end{equation*}
    where $\Omega(t)=\int_0^t\omega(s)ds$. Since the differentiability of $\phi_\omega$ at $x$ yields the derivative of $a$ at $t=0$, we have that $a$ must solve the following initial value problem
    \begin{equation}\label{eq:a_ivp}
        \begin{cases}
            \ddot{a}(t)=V'(a(t)-\Omega(t)),~t\in (0,T),\\
            a(0)=c_0 x,~\dot{a}(0)=-\phi'_\omega(x).
        \end{cases}
    \end{equation}
    Since $\omega \in L^2(0,T)$, we have that $\Omega$ is H\"{o}lder-$\frac{1}{2}$ continuous by~\eqref{eq:elementary}. Additionally, $V'$ is Lipschitz continuous due to the global bounds on $V''$. Thus, the initial value problem~\eqref{eq:a_ivp} admits a unique solution by the Picard-Lindel\"{o}f Theorem. The uniqueness of $a$ implies the uniquness of $z^*(\cdot,x)$.
\end{proof}

Our next focus is on the qualitative properties of $\phi_\omega$.
\begin{proposition}\label{prp:phi_omega}
    Let $\omega \in L^2(0,T)$ be arbitrary. Then $\phi_\omega$ is a globally Lipschitz and semiconcave function; that is,
    \begin{equation}
        \begin{split}
            |\phi_\omega(x+h)-\phi_\omega(x)|\leq& C|h|,\\
            \phi_\omega(x+h)+\phi_\omega(x-h)-2\phi_\omega(x) \leq& C h^2,\quad \forall x,h \in \mathbb{R}.
        \end{split}
    \end{equation}
    In particular, $\phi_\omega$ is Lebesgue a.e. differentiable in $\mathbb{R}$, and~\eqref{eq:phi_omega_1} admits a unique minimizer for Lebesgue a.e. $x\in \mathbb{R}$.
\end{proposition}
\begin{proof}
    Let $z^*(\cdot,x)$ be a minimizer for~\eqref{eq:phi_omega_1}. Then we have that $z_1(t)=z^*(t,x)+h$ and $z_2(t)=z^*(t,x)-h$ are in general suboptimal trajectories for initial points $x+h$ and $x-h$, respectively. Hence,
    \begin{equation}\label{eq:h_ineq}
        \begin{split}
            \phi_\omega(x+h) \leq &\int_0^T \left[c_0 \frac{\dot{z}^*(t,x)^2}{2}+\dot{z}^*(t,x)\omega(t) +V(z^*(t,x)+h)\right] dt+g(z^*(T,x)+h)\\
            = & \phi_\omega(x)+\int_0^T \left[V(z^*(t,x)+h)-V(z^*(t,x))\right] dt+g(z^*(T,x)+h)-g(z^*(T,x))\\
            \leq & \phi_\omega(x)+C |h|,
        \end{split}
    \end{equation}
    where we used global Lipschitz bounds of $V,g$. Exchanging the roles of $x+h$ and $x$ above, we obtain
    \begin{equation*}
        \phi(x)\leq \phi_\omega(x)+C|h|,
    \end{equation*}
    which completes the Lipschitz bounds for $\phi_\omega$.

    For semiconcavity estimates, we use~\eqref{eq:h_ineq} with $-h$ and obtain
    \begin{equation*}
        \phi_\omega(x-h)\leq \phi_\omega(x)+\int_0^T \left[V(z^*(t,x)-h)-V(z^*(t,x))\right] dt+g(z^*(T,x)-h)-g(z^*(T,x)).
    \end{equation*}
    The combination of this inequality with~\eqref{eq:h_ineq} yields
    \begin{equation*}
        \begin{split}
            \phi_\omega(x+h)+\phi_\omega(x-h)-2\phi_\omega(x)\leq& \int_0^T \left[V(z^*(t,x)+h)+V(z^*(t,x)-h)-2V(z^*(t,x))\right] dt\\
            &+g(z^*(T,x)-h)+g(z^*(T,x)-h)-2g(z^*(T,x))\\
            \leq & C|h|^2,
        \end{split}
    \end{equation*}
    where we used global bounds of $V'',g''$.

    The a.e. differentiability of $\phi_\omega$ and uniqueness of minimizers for~\eqref{eq:phi_omega_1} follows from the Rademacher's theorem and Proposition~\ref{prp:min_uniqueness}, respectively.
\end{proof}

Our final preliminary result on the single-agent optimization problem is the stability of the minimizers of~\eqref{eq:phi_omega_1}.

\begin{proposition}\label{prp:min_stability}
    Let $\omega_n,\omega \in L^2(0,T)$ and $x_n,x\in \mathbb{R}$ be such that
    \begin{itemize}
        \item $\{\omega_n\}$ converges strongly to $\omega$ in $L^2(0,T)$,
        \item $\{x_n\}$ converges to $x$.
    \end{itemize}
    Furthermore, assume that $z_n^* \in H^1(0,T)$ are minimizers of~\eqref{eq:phi_omega_1} for the price $\omega_n$ and initial state $x_n$. Then there exists a $z^*\in H^1(0,T)$, which is a minimizer for~\eqref{eq:phi_omega_1} for the price $\omega$ and initial state $x$, such that up to a subsequence
\begin{itemize}
    \item $\{z_n^*\}$ converges strongly to $z^*$ in $H^1(0,T)$,
    \item $\{z_n^*\}$ converges to $z^*$ in $C([0,T])$.
\end{itemize}
\end{proposition}
\begin{proof}
    From Lemma~\ref{lma:existence_mins} we have that
    \begin{equation*}
        \begin{split}
        \|\dot{z}_n^*\|_{2}\leq & C(\|\omega_n\|_2),\\
        |z_n^*(t')-z_n^*(t)|\leq & C(\|\omega_n\|_2)~\sqrt{|t'-t|},\quad \forall t,t'\in [0,T],\\
        |z_n^*(t)|\leq& |x_n|+ C(\|\omega_n\|_2),\quad \forall t\in [0,T].
    \end{split}
    \end{equation*}
    Furthermore, the convergence of $\{\omega_n\}$ and $\{x_n\}$ implies their boundedness, so $\{z_n\}$ is
    \begin{itemize}
        \item a bounded sequence in $H^1(0,T)$,
        \item a uniformly bounded and equicontinuous sequence in $C([0,T])$.
    \end{itemize}
    Therefore, there exists $z^*\in H^1(0,T)$ such that up to a subsequence
    \begin{itemize}
    \item $\{z_n^*\}$ converges weakly to $z^*$ in $H^1(0,T)$,
    \item $\{z_n^*\}$ converges to $z^*$ in $C([0,T])$.
    \end{itemize}
    We claim that $z^*$ is a minimizer for~\eqref{eq:phi_omega_1}. Firstly, we have that
    \begin{equation*}
        z^*(0)=\lim_{n \to \infty} z^*_n(0)=\lim_{n \to \infty} x_n=x.
    \end{equation*}
    Next, let $z\in H^1(0,T)$ be arbitrary such that $z(t)=x$. Then we have that
    \begin{equation*}
        \begin{split}
            &\int_0^T \left[c_0 \frac{\dot{z}^*_n(t)^2}{2}+\dot{z}^*_n(t)\omega_n(t) +V(z^*_n(t))\right] dt+g(z^*_n(T))\\
            \leq & \int_0^T \left[c_0 \frac{\dot{z}(t)^2}{2}+\dot{z}(t)\omega_n(t) +V(z(t)+x_n-x)\right] dt+g(z(T)+x_n-x)
        \end{split}
    \end{equation*}
    Passing to the limit as $n \to \infty$ and taking into account that
    \begin{equation*}
        \liminf_{n \to \infty} \int_0^T  \dot{z}^*_n(t)^2 dt \geq \int_0^T  \dot{z}^*(t)^2 dt\quad \text{and} \quad \lim_{n \to \infty} \int_0^T \dot{z}^*_n(t)\omega_n(t) dt = \int_0^T \dot{z}^*(t)\omega(t) dt,
    \end{equation*}
    we find that
    \begin{equation*}
        \begin{split}
            &\int_0^T \left[c_0 \frac{\dot{z}^*(t)^2}{2}+\dot{z}^*(t)\omega(t) +V(z^*(t))\right] dt+g(z^*(T))\\
            \leq & \int_0^T \left[c_0 \frac{\dot{z}(t)^2}{2}+\dot{z}(t)\omega(t) +V(z(t))\right] dt+g(z(T)),
        \end{split}
    \end{equation*}
    which means that $z^*$ is a minimizer for~\eqref{eq:phi_omega_1}.

    Next, we denote by
    \begin{equation*}
    \begin{split}
        a_n(t)=z_n^*(t)+\int_0^t \omega_n(s) ds \quad \text{and} \quad a(t)=z^*(t)+\int_0^t \omega(s) ds.
    \end{split}
    \end{equation*}
    From Proposition~\ref{prp:min_uniqueness} we have that $a_n,a \in C^2([0,T])$, and
    \begin{equation*}
        \ddot{a}_n(t)=V'(z_n^*(t))\quad \text{and}\quad \ddot{a}(t)=V'(z^*(t)),\quad t\in [0,T],
    \end{equation*}
    which means that $\{\ddot{a}_n\}$ converges uniformly to $\ddot{a}$ in $C([0,T])$. The latter yields that $\{\dot{a}_n\}$ converges uniformly to $\dot{a}$, which in particular means that $\{\dot{a}_n\}$ converges strongly to $\dot{a}$ in $L^2(0,T)$. Since
    \begin{equation*}
        \dot{z}^*_n=\dot{a}_n-\omega_n \quad \text{and}\quad \dot{z}^*=\dot{a}-\omega,
    \end{equation*}
    we conclude that $\{\dot{z}_n^*\}$ converges strongly to $\dot{z}^*$ in $L^2(0,T)$, and so $\{z_n^*\}$ converges strongly to $z^*$ in $H^1(0,T)$.
\end{proof}

\begin{remark}\label{rmk:measurable_selection}
    Fix an $\omega\in L^2(0,T)$. Proposition~\ref{prp:phi_omega} implies that~\eqref{eq:phi_omega_1} admits a unique minimizer $z^*(\cdot,x)$ for Lebesgue a.e. $x\in \mathbb{R}$. Furthermore, Proposition~\ref{prp:min_stability} implies that the minimizers of~\eqref{eq:phi_omega_1} are stable under perturbations of the initial state $x\in \mathbb{R}$. Hence, for points $x\in \mathbb{R}$ such that~\eqref{eq:phi_omega_1} admits multiple minimizers, we can choose $z^*(\cdot,x)$ in such a way that the map $(t,x)\mapsto z^*(t,x)$ is Borel measurable~\cite{aubin90set}. We denote one such selection by $z^*_\omega(t,x)$.
\end{remark}

\section{Analysis of the variational problem for the price}\label{sec:I}

\begin{proposition}\label{prp:J}
    Nonlinear functional $J$ defined in~\eqref{eq:J} is everywhere finite concave functional in $L^2(0,T)$ such that
    \begin{equation}\label{eq:J_upper_bound}
        J[\omega]\leq C-C\|\omega\|_2^2,
    \end{equation}
    where $C>0$ is a constant that depends only on problem data.
    
    Furthermore, $J$ is everywhere Fr\'{e}chet differentiable with
    \begin{equation*}
        \delta J[\omega](t)=\int_{\mathbb{R}} \dot{z}^*_{\omega}(t,x)\rho_0(x)dx, \quad t\in (0,T),
    \end{equation*}
    where $z^*_\omega$ is a Borel measurable selection of minimizing trajectories (Remark~\ref{rmk:measurable_selection}).
\end{proposition}
\begin{proof}
    Let $\omega \in L^2(0,T)$ be arbitrary. By Proposition~\ref{prp:phi_omega} we have that $\phi_\omega$ is a Lipschitz continuous function. Since $\rho_0$ is compactly supported, we obtain that $\phi_\omega$ is bounded on the support of $\rho_0$, and so $J[\omega]$ is finite.

    Let $z^*_\omega$ be a Borel measurable selection of minimizing trajectories. We denote by
    \begin{equation}\label{eq:p}
        p_\omega(t)=\int_{\mathbb{R}} \dot{z}^*_{\omega}(t,x)\rho_0(x)dx,\quad t\in (0,T).
    \end{equation}
    Note that $p_\omega$ is well defined and does not depend on a particular selection $z^*_\omega$. Indeed, there is a unique choice for $z^*_\omega(\cdot,x)$ for $\rho_0$ a.e. $x$ since the minimizers of~\eqref{eq:phi_omega_1} are unique for Lebesgue a.e. $x$ (Proposition~\ref{prp:phi_omega}), and $\rho_0$ is absolutely continuous with respect to the Lebesgue measure.

    Furthermore, bounds~\eqref{eq:min_bounds} yield
    \begin{equation}\label{eq:p_omega_bound}
        \begin{split}
            \int_0^T p^2_\omega(t) dt =&\int_0^T \left( \int_{\mathbb{R}} \dot{z}^*_{\omega}(t,x)\rho_0(x)dx\right)^2 dt\leq \int_0^T \int_{\mathbb{R}} \dot{z}^*_{\omega}(t,x)^2 \rho_0(x)dx dt\\
            = &  \int_{\mathbb{R}} \rho_0(x)dx \int_0^T  \dot{z}^*_{\omega}(t,x)^2  dt \leq C(\|\omega\|_2),
        \end{split}
    \end{equation}
    so $p_\omega \in L^2(0,T)$. Let $\hat{\omega}\in L^2(0,T)$ be arbitrary. Then we have that $z^*_\omega$ are in general suboptimal for the price $\hat{\omega}$, so
    \begin{equation*}
        \begin{split}
            \phi_{\hat{\omega}}(x)\leq& \int_0^T \left[c_0 \frac{\dot{z}^*_{\omega}(t,x)^2}{2}+\dot{z}^*_{\omega}(t,x)\hat{\omega}(t) +V(z^*_{\omega}(t,x))\right] dt+g(z^*_{\omega}(T,x))\\
            =&\int_0^T \left[c_0 \frac{\dot{z}^*_{\omega}(t,x)^2}{2}+\dot{z}^*_{\omega}(t,x)\omega(t) +V(z^*_{\omega}(t,x))\right] dt+g(z^*_{\omega}(T,x))\\
            &+\int_0^T \dot{z}^*_{\omega}(t,x)\left(\hat{\omega}(t)-\omega(t)\right) dt\\
            =&\phi_\omega(x)+\int_0^T \dot{z}^*_{\omega}(t,x)\left(\hat{\omega}(t)-\omega(t)\right) dt.
        \end{split}
    \end{equation*}
    Integrating this previous inequality with respect to $\rho_0$, we obtain
    \begin{equation}\label{eq:J_superdiff}
            J[\hat{\omega}]\leq J[w]+\langle p_\omega, \hat{\omega}-\omega \rangle,
    \end{equation}
    which implies that $J$ is concave.

    Furthermore, for every $z\in H^1(0,T)$ such that $z(0)=x$ we have that
    \begin{equation*}
        \begin{split}
            &\int_0^T V(z(t)) dt+g(z(T)) \leq T\cdot V(x)+g(x)+C \int_0^T |z(t)-x|dt+C|z(T)-x|\\
            \leq & T\cdot V(x)+g(x)+ C \int_0^T |\dot{z}(t)| dt \leq T\cdot V(x)+g(x)+C \|\dot{z}\|_2.
        \end{split}
    \end{equation*}
    Hence, taking $z(t)=x-\frac{1}{c_0}\int_0^t \omega(s)ds$, we obtain
    \begin{equation*}
        \begin{split}
            \phi_\omega(x)\leq & \int_0^T \left[c_0 \frac{\dot{z}(t)^2}{2}+\dot{z}(t)\omega(t) +V(z(t))\right] dt+g(z(T))\\
            \leq & -\frac{\|\omega\|_2^2}{2c_0}+ C\|\omega\|_2+ T\cdot V(x)+g(x)\leq -C \|\omega\|_2^2+C+T\cdot V(x)+g(x).
        \end{split}
    \end{equation*}
    Integrating this previous inequality with respect to $\rho_0$, we obtain 
    \begin{equation*}
        J[\omega]=\int_{\mathbb{R}} \phi_\omega(x) \rho_0(x)dx \leq -C \|\omega\|_2^2+C + \int_{\mathbb{R}}\left( T\cdot V(x)+g(x)\right)\rho_0(x)dx= -C \|\omega\|_2^2+C.
    \end{equation*}
    To complete the proof, we need to show that $p_\omega$ is the Fr\'{e}chet derivative of $J[\omega]$. Using~\eqref{eq:J_superdiff}, we have that
    \begin{equation*}
        0\geq J[\hat{\omega}]-J[\omega]-\langle p_\omega, \hat{\omega}-\omega \rangle \geq \langle p_{\hat{\omega}}-p_\omega, \hat{\omega}-\omega \rangle \geq -\|p_{\hat{\omega}}-p_\omega\|_2~\|\hat{\omega}-\omega\|_2,
    \end{equation*}
    for all $\hat{\omega},\omega \in L^2(0,T)$. Thus, we are done if we prove
    \begin{equation*}
        \lim_{\hat{\omega}\to \omega} \|p_{\hat{\omega}}-p_\omega\|_2=0.
    \end{equation*}
    To this end, let $\omega_n,\omega \in L^2(0,T)$ be such that $\{\omega_n\}$ converges strongly to $\omega$. Arguing similarly to~\eqref{eq:p_omega_bound}, we obtain that
    \begin{equation*}
        \int_0^T (p_{\omega_n}(t)-p_{\omega}(t))^2 dt \leq \int_{\mathbb{R}} \rho_0(x)dx \int_0^T (\dot{z}^*_{\omega_n}(t,x)-\dot{z}^*_\omega(t,x))^2 dt.
    \end{equation*}
    From Proposition~\ref{prp:phi_omega} we have that $z^*_{\omega}(\cdot,x)$ is the unique minimizer for~\eqref{eq:phi_omega_1} for $\rho_0$ a.e. $x\in \mathbb{R}$. For all such $x$, Proposition~\ref{prp:min_stability} yields that $z^*_{\omega_n}(\cdot,x)$ converges strongly to $z^*_{\omega}(\cdot,x)$ in $H^1(0,T)$, so
    \begin{equation*}
      \lim_{n \to \infty}  \int_0^T (\dot{z}^*_{\omega_n}(t,x)-\dot{z}^*_\omega(t,x))^2 dt=0,\quad \rho_0~\text{a.e.}
    \end{equation*}
    Furthermore, bounds~\eqref{eq:min_bounds} yield that
    \begin{equation*}
        \int_0^T (\dot{z}^*_{\omega_n}(t,x)-\dot{z}^*_\omega(t,x))^2 dt \leq \int_0^T 2(\dot{z}^*_{\omega_n}(t,x)^2+\dot{z}^*_\omega(t,x)^2) dt\leq C(\|\omega_n\|_2)+C(\|\omega\|_2).
    \end{equation*}
    Since $\{\omega_n\}$ converges strongly to $\omega$, we have that $\{\omega_n\}$ is bounded, and the dominated convergence theorem yields 
    \begin{equation*}
        \lim_{n \to \infty}\int_{\mathbb{R}} \rho_0(x)dx \int_0^T (\dot{z}^*_{\omega_n}(t,x)-\dot{z}^*_\omega(t,x))^2 dt=0,
    \end{equation*}
    and so
    \begin{equation*}
        \lim_{n \to \infty} \int_0^T (p_{\omega_n}(t)-p_{\omega}(t))^2 dt=0.
    \end{equation*}
\end{proof}

\begin{theorem}\label{thm:equilibrium=min}
    Let $Q\in L^2(0,T)$. Then $\omega^*\in L^2(0,T)$ is an equilibrium price if and only if $\omega^* \in \arg\min\limits_{\omega} I[\omega]$, where $I$ is the functional defined in~\eqref{eq:I}.
\end{theorem}
\begin{proof}
    Let $\omega^* \in L^2(0,T)$. By Proposition~\ref{prp:phi_omega} optimal trading strategies (trajectories) are unique for $\rho_0$ a.e. $x\in \mathbb{R}$. Hence, $\rho_0$ a.e. we have that $\alpha^*_{\omega^*}(t,x)=\dot{z}^*_{\omega^*}(t,x)$, and  
    \begin{equation*}
        \int_{\mathbb{R}}\alpha^*_{\omega^*}(t,x) \rho_0(x)dx=\int_{\mathbb{R}}\dot{z}^*_{\omega^*}(t,x)\rho_0(x)dx=p_{\omega^*}(t).
    \end{equation*}
    From Proposition~\ref{prp:J} we have that $I$ is a convex Fr\'{e}chet differentiable functional with
    \begin{equation*}
        \delta_\omega I[\omega]=Q-p_\omega.
    \end{equation*}
    Thus, $\omega^*\in \arg \min\limits_{\omega} I[\omega]$ if and only if $\delta_\omega I[\omega^*]=0$ or 
    \begin{equation*}
        Q(t)=p_{\omega^*}(t)=\int_{\mathbb{R}}\alpha^*_{\omega^*}(t,x) \rho_0(x)dx,
    \end{equation*}
    which is the definition of the equilibrium price.
\end{proof}

\begin{theorem}\label{thm:min_existence}
    For all $Q\in L^2(0,T)$, the functional $I$ admits a minimizer. Moreover, for every minimizer $\omega^*$, we have that $\omega^*+c_0 Q \in C^1([0,T])$. 

    In particular, for every supply $Q\in L^2(0,T)$ there exists an equilibrium price $\omega^*$, which is as smooth as $Q$, up to $C^1([0,T])$.
\end{theorem}
\begin{proof}
    From proposition~\ref{prp:J} we have that 
    \begin{equation*}
        I[\omega]=\int_0^T Q(t)\omega(t)dt-J[\omega]\geq -\|Q\|_2\|\omega\|_2+C\|\omega\|_2^2-C\geq C\|\omega\|_2^2-C,\quad \forall w\in L^2(0,T),
    \end{equation*}
    and $\inf\limits_\omega I[\omega]>-\infty$. Furthermore, let $\{\omega_n\}$ be a minimizing sequence; that is,
    \begin{equation*}
        \lim_{n \to \infty} I[\omega_n]=\inf_\omega I[\omega].
    \end{equation*}
    Then we have that $\{I[\omega_n]\}$ is bounded, and
    \begin{equation*}
        C\|\omega_n\|_2^2-C\leq I[\omega_n] \leq C \Longrightarrow \|\omega_n\|_2 \leq C,\quad \forall n\geq 1.
    \end{equation*}
    Hence, there exists a $\omega^* \in L^2(0,T)$ such that (up to a seubsequence) $\{\omega_n\}$ converges weakly to $\omega^*$, and the convexity of $I$ yields that 
    \begin{equation*}
        I[\omega^*]\leq I[\omega_n] - \langle \delta_\omega I[\omega^*], \omega_n -\omega^* \rangle_2 \Longrightarrow I[\omega^*] \leq \lim_{n \to \infty } I[\omega_n]=\inf_\omega I[\omega].
    \end{equation*}
    Thus, $\omega^* \in \arg\min\limits_{\omega} I[\omega]$.

    Now let $\omega^* \in \arg\min\limits_{\omega} I[\omega]$ be an arbitrary minimizer. Then we have that
    \begin{equation*}
        \delta_\omega I[\omega^*]=Q(t)-\int_{\mathbb{R}}\dot{z}^*_{\omega^*}(t,x)\rho_0(x)dx=0.
    \end{equation*}
    Denote by $a(t,x)=c_0 z^*_{\omega^*}(t,x)+\int_0^t \omega^*(s) ds$. Then from Proposition~\ref{prp:min_uniqueness} we have that $a(\cdot,x)$ is twice continuously differentiable for all $x\in \mathbb{R}$, and 
    \begin{equation*}
    \begin{split}
        |\ddot{a}(t,x)|=&|V'(z^*_{\omega^*}(t,x))|\leq C,\quad \forall t\in [0,T],\\
|\dot{a}(T,x)|=&|g'(z^*_{\omega^*}(t,x))|\leq C,        
    \end{split}
    \end{equation*}
    where we used the global bounds on $V',g'$. Hence, the function $\int_{\mathbb{R}} \dot{a}(\cdot,x) \rho_0(x)dx$  is continuously differentiable with
    \begin{equation*}
        \frac{d}{dt}\int_{\mathbb{R}} \dot{a}(t,x) \rho_0(x)dx= \int_{\mathbb{R}} \ddot{a}(t,x) \rho_0(x)dx= \int_{\mathbb{R}} V'(z^*_{\omega^*}(t,x)) \rho_0(x)dx,\quad t\in [0,T].
    \end{equation*}
    But
    \begin{equation*}
        \int_{\mathbb{R}} \dot{a}(t,x) \rho_0(x)dx=c_0 \int_{\mathbb{R}} \dot{z}^*_{\omega^*}(t,x) \rho_0(x)dx+\omega^*(t)=c_0 Q(t)+\omega^*(t),
    \end{equation*}
    which completes the proof.
\end{proof}

\section{A primal-dual algorithm}\label{sec:alg}

Using the definition~\eqref{eq:phi_omega} of $\phi_\omega$, we first formulate~\eqref{eq:mfg_var} as a saddle point problem. More specifically, we have that
\begin{equation}\label{eq:saddle}
	\begin{split}
		&\inf_\omega \int_0^T \omega(t) Q(t) dt-\int_{\mathbb{R}} \phi_\omega(x) \rho_0(x)dx=\inf_\omega \sup_{\dot{z}_x=\alpha_x,z_x(0)=x} \int_0^T \omega(t) Q(t) dt\\
		&- \int_{\mathbb{R}} \left[ \int_0^T\left( L(z_x(t),\alpha_x(t))+\alpha_x(t)\omega(t) \right) dt+g(z_x(T))\right] \rho_0(x)dx.
	\end{split}
\end{equation}
Next, we discretize~\eqref{eq:saddle} as follows. We consider a uniform time discretization $0=t_0<t_1 < t_2 < \ldots < t_N = T$, and denote the discretized states, price, and supply by 
\begin{equation*}
\begin{split}
    \mathbf{z}_x =&(z_x[0],z_x[1], z_x[2], \ldots, z_x[N])\in \mathbb{R}^{N+1},\\
    \bm{\omega} =& (\omega[0],\omega[1], \omega[2], \ldots, \omega[N-1])\in \mathbb{R}^N \\
    \bm{Q} =& (Q[0],Q[1], Q[2], \ldots, Q[N-1])\in \mathbb{R}^N,
\end{split}
\end{equation*}
where
\begin{equation*}
    z_x[l]=z_x \left(\frac{lT}{N} \right),\quad \omega[l]=\omega \left(\frac{lT}{N} \right), \quad Q[l]=Q \left(\frac{lT}{N} \right),\quad 0\leq l \leq N,
\end{equation*}
and the subscript $x$ denotes the trajectories with the initial state $x\in \mathbb{R}$.

Since the inner optimization problem in~\eqref{eq:saddle} is a trajectory-optimization problem, we use simple but effective direct transcription approach~\cite{betts98survey} for solving it:
\begin{equation}
	\label{eq:saddle_discretized}
	\begin{split}
		\inf_{\bm{\omega}} \sup_{\bm{\alpha}} \frac{T}{N} \sum_{l=0}^{N-1}  {\omega}[l] {Q}[l]- \frac{1}{M}\sum_{m=1}^{M}\left[ \frac{T}{N} \sum_{l=0}^{N-1} \left(L( {z}^{{\alpha}}_{x_m}[l],{\alpha}_{x_m}[l])+ {\alpha}_{x_m}[l] {\omega}[l]\right) +g({z}^{{\alpha}}_{x_m}[N]) \right],
	\end{split}
\end{equation}
where
\begin{equation*}
	\bm{\alpha} = 
	\begin{bmatrix}
		{\alpha}_{x_1}[0] &  {\alpha}_{x_1}[1] & \ldots & {\alpha}_{x_1}[N-1]
		\\
		{\alpha}_{x_2}[0] &  {\alpha}_{x_2}[1] & \ldots & {\alpha}_{x_2}[N-1]
		\\
		\vdots
		\\
		{\alpha}_{x_M}[0] &  {\alpha}_{x_M}[1] & \ldots & {\alpha}_{x_M}[N-1]
	\end{bmatrix}\in \mathbb{R}^{M \times N},
\end{equation*}
is the discretized control, each row representing the controls for the trajectory defined by an initial state $x_m$, and
\begin{equation*}
	{z}^{{\alpha}}_{x_m}[l+1]={z}^{{\alpha}}_{x_m}[l] + \frac{T}{N} {\alpha}_{x_m}[l], \quad {z}^{{\alpha}}_{x_m}[0] = x_m,\quad 0\leq l \leq N-1,\quad 1\leq m \leq M.
\end{equation*}
Moreover, $x_1, x_2, \ldots, x_M$ are samples of initial states drawn from $\rho_0$.

Various saddle-point optimization algorithms can be used to solve~\eqref{eq:saddle}. While we used an Euler discretization of the dynamics, any other method could also be used, e.g., RK4. As in~\cite{nursaude18}, we use a version of the primal-dual hybrid gradient (PDHG) algorithm of Chambolle and Pock~\cite{champock11,champock16} to solve~\eqref{eq:saddle_discretized}. Denoting by
\begin{equation*}
	\begin{split}
		\mathcal{L}(\bm{\omega}, \bm{\alpha})= \frac{T}{N} \sum_{l=0}^{N-1}  {\omega}[l] {Q}[l]- \frac{1}{M}\sum_{m=1}^{M}\left[ \frac{T}{N}\sum_{l=0}^{N-1} \left(L( {z}^{{\alpha}}_{x_m}[l],{\alpha}_{x_m}[l])+ {\alpha}_{x_m}[l] {\omega}[l]\right) +g({z}^{{\alpha}}_{x_m}[N]) \right],
	\end{split}
\end{equation*}
\eqref{eq:saddle_discretized} reduces to
\begin{equation}\label{eq:saddle_L}
	\inf_{\bm{\omega}} \sup_{\bm{\alpha}} \mathcal{L}(\bm{\omega},\bm{\alpha}).
\end{equation}
The PDHG algorithm comprises of the following updates
\begin{equation*}
	\begin{split}
		&\bm{\alpha}^{(k+1)} = \arg\max_{\bm{\alpha}}\mathcal{L}(\bm{\omega}^{(k)},\bm{\alpha}) -\frac{\|\bm{\alpha}-\bm{\alpha}^{(k)}\|^2}{2\tau_{\alpha}},
		\\
		&\overline{\bm{\alpha}}^{(k+1)} = 2\bm{\alpha}^{(k+1)} - \bm{\alpha}^{(k)}, 
		\\ 
		&\bm{\omega}^{(k+1)} = \arg\min_{\bm{\omega}} \mathcal{L}(\bm{\omega},\overline{\bm{\alpha}}^{(k+1)})+\frac{\|\bm{\omega}-\bm{\omega}^{(k)}\|^2}{2\tau_\omega},
	\end{split}
\end{equation*}
where $\tau_\alpha,\tau_\omega>0$ are suitably chosen step sizes, and $(\bm{\omega}^{(0)},\bm{\alpha}^{(0)})$ are chosen arbitrarily. To be consistent with the continuous problem, we use the norms
\begin{equation*}
	\|\bm{\alpha}\|^2=\frac{T}{MN}\sum_{m=1}^M \sum_{l=0}^{N-1} {\alpha}_{x_m}[l]^2=\frac{T}{MN}\|\bm{\alpha}\|^2_2,\quad \|\bm{\omega}\|^2=\frac{T}{N} \sum_{l=0}^{N-1} {\omega}[l]^2=\frac{T}{N}\|\bm{\omega}\|^2_2,
\end{equation*}
so the updates read as
\begin{equation}\label{eq:original_PDHG}
	\begin{split}
		&\bm{\alpha}^{(k+1)} = \arg\max_{\bm{\alpha}}\mathcal{L}(\bm{\omega}^{(k)},\bm{\alpha}) -\frac{T}{2MN\tau_{\alpha}}\|\bm{\alpha}-\bm{\alpha}^{(k)}\|_2^2,
		\\
		&\overline{\bm{\alpha}}^{(k+1)} = 2\bm{\alpha}^{(k+1)} - \bm{\alpha}^{(k)}, 
		\\ 
		&\bm{\omega}^{(k+1)} = \arg\min_{\bm{\omega}} \mathcal{L}(\bm{\omega},\overline{\bm{\alpha}}^{(k+1)})+\frac{T}{2N\tau_\omega}\|\bm{\omega}-\bm{\omega}^k\|_2^2.
	\end{split}
\end{equation}
Recall that the original version of the PDHG algorithm in~\cite{champock11} requires
\begin{enumerate}
    \item convex and proximable $\bm{\omega} \mapsto \mathcal{L}(\bm{\omega},\bm{\alpha})$,
    \item concave and proximable $\bm{\alpha} \mapsto \mathcal{L}(\bm{\omega},\bm{\alpha})$,
    \item a bilinear coupling between $\bm{\omega},\bm{\alpha}$.
\end{enumerate}
The term proximable refers to the ease of solving the first and third optimization steps during the update~\eqref{eq:original_PDHG}; that is, either these steps admit closed-form solutions or can be solved efficiently with a high precision~\cite{champock11},~\cite[Section 6]{parikhboyd14}. In our case, one of the three conditions above is invalid: $\bm{\alpha} \mapsto \mathcal{L}(\bm{\omega},\bm{\alpha})$ is not concave and proximable in general.

Hence, we replace the proximal optimization step in $\bm{\alpha}$ in~\eqref{eq:original_PDHG} by a gradient ascent step:
\begin{equation}\label{eq:algo}
	\begin{split}
		&\bm{\alpha}^{(k+1)} = \bm{\alpha}^{(k)} + \frac{\tau_{\alpha}MN}{T}~\nabla_{\bm{\alpha}} \mathcal{L}(\bm{\omega}^{(k)}, \bm{\alpha}^{(k)}),
		\\
		&\overline{\bm{\alpha}}^{(k+1)} = 2\bm{\alpha}^{(k+1)} - \bm{\alpha}^{(k)}, 
		\\ 
		&\bm{\omega}^{(k+1)} = \arg\min_{\bm{\omega}} \mathcal{L}(\bm{\omega},\overline{\bm{\alpha}}^{(k+1)})+\frac{T}{2N\tau_\omega}\|\bm{\omega}-\bm{\omega}^k\|_2^2.
	\end{split}
\end{equation}
Thus, we arrive at a version of PDHG where smooth parts of $\bm{\alpha}\mapsto \mathcal{L}(\bm{\omega},\bm{\alpha})$ can be linearized within proximal steps~\cite[Algorithm 1]{champock16}. Still, the convergence proof in~\cite{champock16} requires the function $\bm{\alpha}\mapsto \mathcal{L}(\bm{\omega},\bm{\alpha})$ to be concave, which is not the case in general for us.

Despite the non-concavity of $\bm{\alpha}\mapsto \mathcal{L}(\bm{\omega},\bm{\alpha})$, our tests demonstrate a very robust performance of~\eqref{eq:algo}. In particular, we did not encounter convergence issues, did not have to perform hyperparameter tuning, and can match analytic solutions, where available, with high precision (see Section~\ref{sec:experiments}). We have previously observed similar performance for a closely related algorithm for numerically solving mean-field games with nonlocal interactions~\cite{nursaude18}.

The gradient ascent step in $\bm{\alpha}$ is implemented via automatic differentiation, whereas the proximal step in $\bm{\omega}$ admits a closed-form solution. Indeed, the first order optimality conditions with respect to $\bm{\omega}$ yield
\begin{equation*}
	\begin{split}
		{\omega}^{(k+1)}[l]={\omega}^{(k)}[l]+\tau_\omega  \left(\frac{1}{M}\sum_{m=1}^M \overline{{\alpha}}^{(k+1)}_{x_m}[l]-Q[l]\right) ,\quad 0\leq l \leq N-1,
	\end{split}
\end{equation*}
which simplifies~\eqref{eq:algo} to the Algorithm~\ref{alg:main} below.
\begin{algorithm}[H]
\caption{\small{Primal-dual algorithm for a mean-field game price formation model}}
\label{alg:main}
\begin{algorithmic}[1]
\Require Lagrangian $L$, cost function $g$, initial density $\rho_0$, control space $Q$; step sizes $\tau_\alpha,\tau_\omega > 0$.
\State Initialize $\bm{\alpha}^{(0)}$, $\overline{\bm{\alpha}}^{(0)}$, $\bm{\omega}^{(0)}$.
\For{$k = 0, 1, 2, \dots$}
    \State $\bm{\alpha}^{(k+1)} \gets \bm{\alpha}^{(k)} + \dfrac{\tau_{\alpha} MN}{T} \nabla_{\bm{\alpha}} \mathcal{L}(\bm{\omega}^{(k)}, \bm{\alpha})$ (gradient computed via \textbf{automatic differentiation})
    \State $\overline{\bm{\alpha}}^{(k+1)} \gets 2\bm{\alpha}^{(k+1)} - \bm{\alpha}^{(k)}$
    \For{$l = 0$ to $N-1$}
        \State $\omega^{(k+1)}[l] \gets \omega^{(k)}[l] + \tau_\omega \left( \dfrac{1}{M} \sum_{m=1}^{M} \overline{\alpha}^{(k+1)}_{x_m}[l] - Q[l] \right)$
    \EndFor
\EndFor
\end{algorithmic}
\end{algorithm}

\section{Numerical experiments}\label{sec:experiments}

In this section, we perform numerical experiments to examine the performance of our algorithm. We consider two types of examples. First, we consider $V$ and $g$ that are quadratic; that is,
\begin{equation}\label{eq:g_V_quad}
    V(z)=\frac{r_1}{2}(z-y_1)^2,\quad g(z)=\frac{r_2}{2}(z-y_2)^2,\quad z\in \mathbb{R},
\end{equation}
for some $r_1,r_2 \geq 0$, and $y_1,y_2 \in \mathbb{R}$. For such $V,g$, the equilibrium price $\omega^*$ and optimal trajectories $z^*_{\omega^*}$ admit analytic solutions, which is convenient for testing the validity our algorithm.

More specifically, the equilibrium price in this case is given by

\begin{equation}\label{eq:omega_analytic_Q}
    \omega^*(t)=r_2(y_2-\bar{x}_0) +r_1 (T-t)(y_1-\bar{x}_0)-c_0 Q(t)+\int_0^T \left[-(r_2+r_1 T)+r_1 \max\{s,t\} \right]Q(s)ds,
\end{equation}    
where
\begin{equation}\label{eq:initial_mean}
    \bar{x}_0=\int_{\mathbb{R}} x \rho_0(x) dx.
\end{equation}

Once the equilibrium price, $\omega^*$, is known, one can find the optimal trajectories of the agents as follows. Denoting by
\begin{equation}\label{eq:k}
    k=\sqrt{\frac{r_1}{c_0}},
\end{equation}
and
\begin{equation}\label{eq:B}
B=r_2(y_2-y_1)+\int_0^T \omega^*(s) \left[\frac{r_2}{c_0}\cosh k(T-s)+k \sinh k(T-s)\right] ds,
\end{equation}
we have that
\begin{equation}\label{eq:z_analytic_r1=0}
    z^*_{\omega^*}(t,x)=x+\frac{B-r_2(x-y_1)}{c_0+r_2 T}t-\frac{1}{c_0}\int_0^t \omega^*(s) ds,
\end{equation}
when $r_1=0$, and
\begin{equation}\label{eq:z_analytic_r1<>0}
    \begin{split}
        z^*_{\omega^*}(t,x)=&y_1+(x-y_1) \cosh kt+ \frac{B-(x-y_1)(c_0 k \sinh kT+r_2 \cosh kT)}{c_0 k \cosh kT+r_2 \sinh kT} \sinh k t\\
        &-\frac{1}{c_0} \int_0^t \omega^*(s) \cosh k(t-s) d s,
\end{split}
\end{equation}
when $r_1\neq 0$. Appendix~\ref{apdx:analytic} contains more details about the derivation of these expressions.

Next, we consider $V$ and $g$ that are double-well potentials; that is,
\begin{equation}\label{eq:g_V_double_well}
    V(z)=\frac{r_3}{2}(z-y_3)^2(z-y_4)^2,\quad g(z)=\frac{r_4}{2}(z-y_5)^2(z-y_6)^2,\quad z\in \mathbb{R},
\end{equation}
for some $r_3,r_4 \geq 0$, and $y_3,y_4,y_5,y_6 \in \mathbb{R}$. Such $V,g$ have two local minima attracting the agents.

We implement the algorithm in Python and utilize PyTorch~\cite{paszke2017automatic} due to its adeptness at handling tensor operations and facilitating automatic differentiation, which are critical for our optimization tasks. We test two types of supply functions $Q$: sinusoidal and non-smooth. For the algorithm, we initialize ${\omega}$ and ${\alpha}$ randomly from a normal distribution. For the hyperparameters, we choose the number of agents $M = 100$, and the number of time steps $N = 1000$ in the time interval $[0,T]=[0,1]$. Moreover, the control penalty weight $c_0$ is chosen to be $1$. 

Furthermore, we assume that the $M$ agents start at different positions, $\{x_m\}_{m=1}^M$, evenly distributed across the interval $[0,1]$, thereby assuming that $\rho_0$ is the uniform probability measure on the interval $[0,1]$. The parameters $\{y_i\}$ are chosen as follows: $y_1=y_2=0$; $y_3=y_5=0.25$; $y_4=y_6=0.75$. In all cases, we stop after 10,000 iterations of the Algorithm~\ref{alg:main}.

The code to reproduce the examples in this section is available at \href{https://github.com/lnurbek/price-mfg-solver}{https://github.com/lnurbek/price-mfg-solver}.

\subsection{Case One: \(r_1 = 0\), \(r_2 \neq 0\)}

In our first experiment, we take
\begin{equation*}
\begin{split}
    V(z)=0,\quad g(z)=5(z-y_2)^2=5z^2,
\end{split}
\end{equation*}
which corresponds to choosing $r_1 = 0$ and $r_2 = 10$ in~\eqref{eq:g_V_quad}. Additionally, we assume that $Q(t)=\sin 10t$.

\begin{figure}[h]
    \centering
    \begin{subfigure}[t]{0.48\textwidth}
        \centering
        \includegraphics[width=\textwidth]{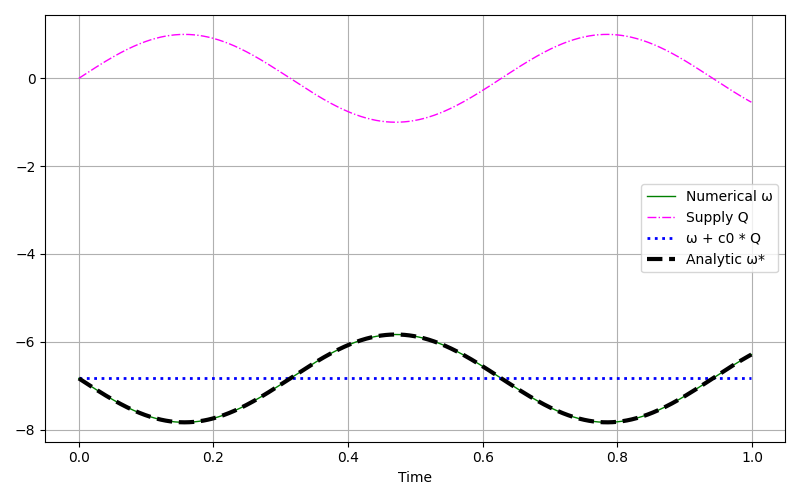}
        \caption{Comparison of numerical and analytic prices: the numerical trajectory (green, solid) closely matches the analytic solution (black, dashed). Also shown are the exogenous supply $Q$ (magenta, dashed and dotted) and the combination curve $\omega+c_0 Q$ (blue, dotted)}
        \label{subfig:case1_w}
    \end{subfigure}
    \hfill
    \begin{subfigure}[t]{0.48\textwidth}
        \centering
        \includegraphics[width=\textwidth]{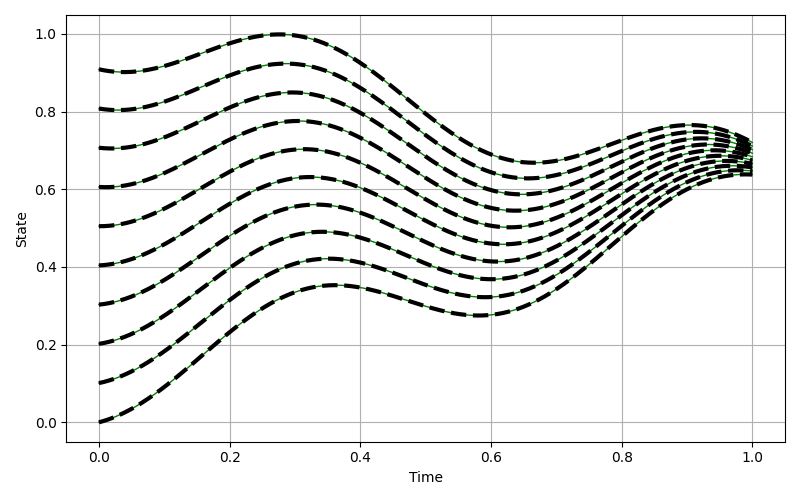}
        \caption{Evolution of state trajectories under the computed equilibrium price. Numerical solutions (green, solid) align with the analytic benchmark (black, dashed), confirming the accuracy of the method.}\label{subfig:case1_z}
    \end{subfigure}
    \caption{Equilibrium price and trajectories for the spatial potential $V(z)=0$, terminal cost $g(z)=5z^2$, and supply function $Q(t)=\sin 10t$.}
    \label{fig:case1}
\end{figure}

In Figure~\ref{fig:case1}, we compare our result against the analytic solution. In particular, Figure~\ref{subfig:case1_w} shows the evolution of the price function $\omega(t)$ over time. The green solid line represents the numerical solution obtained by our method, while the black dashed line indicates the analytic solution. The two curves overlap with an $l_\infty$ error of $1.33\times10^{-14}$, suggesting that our method accurately captures the pricing behavior predicted by theory. The sinusoidal structure of $\omega(t)$ reflects the oscillatory nature of the input signal $Q(t)$; in particular, the trends of $Q(t)$ and $\omega(t)$, or $\omega^*(t)$, are opposite, which is consistent with the real-life economic intuition that price decreases as supply increases.

Furthermore, in Figure~\ref{subfig:case1_z}, we investigate the trajectories $z(t)$ of a sample of $10$ agents. The trajectory of each agent is shown with a pair of lines: a green solid one for the numerical solution and a dashed black one for the analytic one. These trajectories illustrate how agents adjust their states over time under the influence of the price $\omega(t)$. The $l_\infty$ error of $1.29\times10^{-14}$ validates the consistency of the numerical results with the theoretical prediction.

\subsection{Case Two: \(r_1 \neq 0\), \(r_2 = 0\)}\label{subsec:experiments_case2}

Here, we set $r_1 = 10$, $r_2$ = 0 in~\eqref{eq:g_V_quad}; hence, the terminal cost is deactivated, but there is a running cost $V$:
\begin{equation*}
\begin{split}
    V(z)=5(z-y_1)^2=5z^2,\quad g(z)=0.
\end{split}
\end{equation*}
Additionally, we use a realization of a Wiener process $W_t$ as a supply function; that is, $Q(t)=W_t(o)$ for some realization $o$. The purpose of such choice for $Q$ is to show that our numerical method handles non-smooth supply functions, which are covered by the theory. Indeed, realizations of Wiener processes are H\"{o}lder continuous but not differentiable with probability one, and Theorem~\ref{thm:min_existence} asserts that there should exist an equilibrium price $\omega^*$ such that $\omega^*+c_0 Q \in C^1([0,T])$. Hence, we should recover an $\omega^*$ that has the same regularity as $Q$; that is, H\"{o}lder continuous and non-differentiable.

\begin{figure}[h]
    \centering
    \begin{subfigure}[t]{0.48\textwidth}
        \centering
        \includegraphics[width=\textwidth]{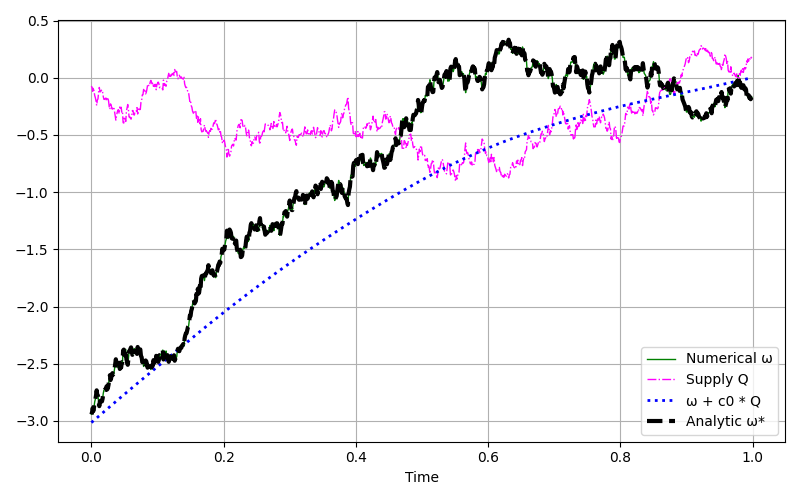}
        \caption{Comparison of numerical and analytic prices: the numerical trajectory (green, solid) closely matches the analytic solution (black, dashed). Also shown are the non-smooth exogenous supply $Q$ (magenta, dashed and dotted) and the \textit{smooth} combination curve $\omega+c_0 Q$ (blue, dotted)}
        \label{subfig:case2_w}
    \end{subfigure}
    \hfill
    \begin{subfigure}[t]{0.48\textwidth}
        \centering
        \includegraphics[width=\textwidth]{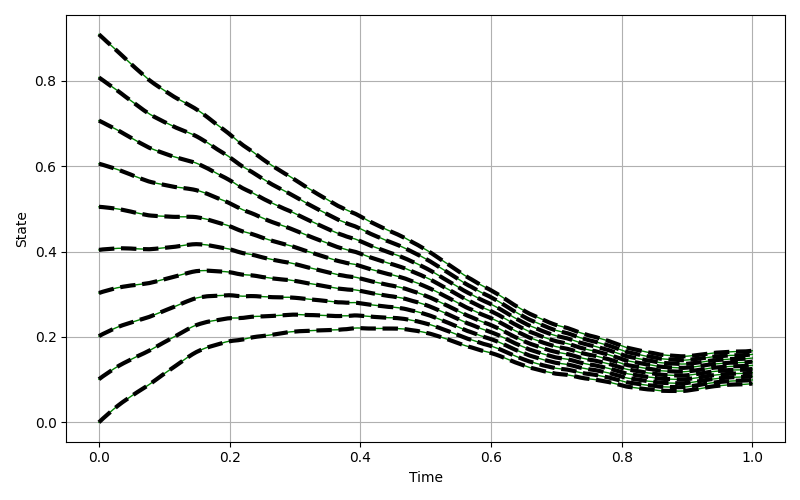}
        \caption{Evolution of state trajectories under the computed equilibrium price. Numerical solutions (green, solid) align with the analytic benchmark (black, dashed), confirming the accuracy of the method.}\label{subfig:case2_z}
    \end{subfigure}
    \caption{Equilibrium price and trajectories for the spatial potential $V(z)=5z^2$, terminal cost function $g(z)=0$, and supply function $Q(t)=W_t(o)$.}\label{fig:case2}
\end{figure}

Figure~\ref{fig:case2} illustrates the results in this case. In particular, Figure~\ref{subfig:case2_w} presents the evolution of the price $\omega(t)$ over time. As before, the green solid curve corresponds to the numerically optimized price, while the black dashed curve represents the analytic solution.

The jagged nature of both curves reflects the non-smoothness of the supply function. Despite the non-smoothness, the numerical and analytic curves align closely with an $l_\infty$ error of $1.33\times10^{-3}$, indicating that our method is robust under non-smooth supply inputs and is aligned with the theoretical predictions. Additionally, the combination curve $\omega+c_0 Q$ (blue, dotted) appears to be smooth as predicted by the theory.

In Figure~\ref{subfig:case2_z}, a sample of 10 agents starts from evenly spaced positions. The consistency between the numerical and analytic solutions is reflected in the $l_\infty$ error of $2.32\times10^{-4}$.

\subsection{Case Three: \(r_3 = 0\), \(r_4 \neq 0\)}\label{subsec:experiments_case3}

In this experiment, we consider a setting where the running cost is inactive and the agents' behavior is only determined by the terminal double-well potential~\eqref{eq:g_V_double_well} with $r_3=0, r_4=50$; that is,
\begin{equation*}
    V(z)=0,\quad g(z)=25(z-y_5)^2(z-y_6)^2=25(z-0.25)^2(z-0.75)^2.
\end{equation*}
Furthermore, we take again a sinusoidal supply function $Q(t)=\sin 10t$. 

Figure~\ref{subfig:case3_w} shows the evolution of the price, which exhibits a smooth oscillatory behavior in response to the periodic fluctuations of the supply function. As before, the price and supply have opposite trends.

In Figure~\ref{subfig:case3_z}, we plot the trajectories given a sample of $10$ agents. Since there is no running cost to penalize their actions throughout the entire time interval, agents primarily respond to the terminal cost. The result is a clear bifurcation: agents converge into two distinct clusters -- one near $y_5=0.25$, the other close to $y_6=0.75$, reflecting the two wells of $g(z)$. Furthermore, although agents are distributed uniformly at the initial time, there are more agents clustered near $y_6=0.75$ at the final time. Indeed, the supply initially increases encouraging the agents to move closer to $y_6=0.75$. Hence, when the supply starts wearing off, more agents prefer to move towards the closer well centered at $y_6=0.75$.

\begin{figure}[h]
    \centering
    \begin{subfigure}[t]{0.48\textwidth}
        \centering
        \includegraphics[width=\textwidth]{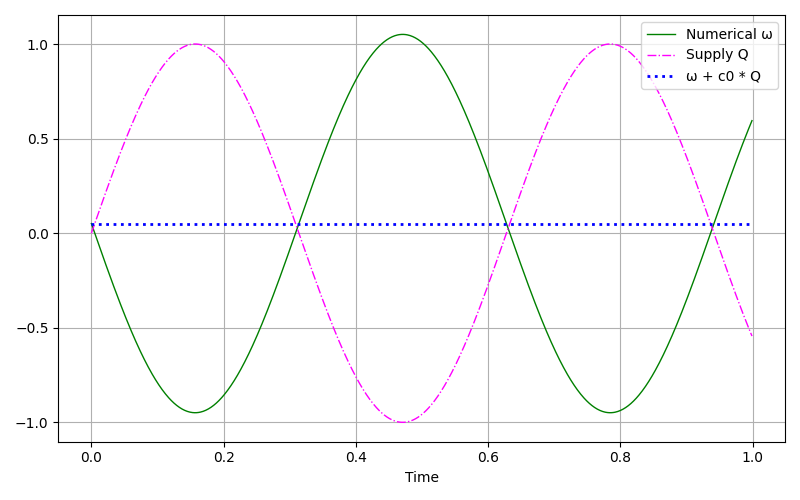}
        \caption{Numerical price trajectory $\omega(t)$ (green, solid) alongside exogenous supply $Q$ (magenta, dashed and dotted) and the combination curve $\omega+c_0 Q$ (dotted, blue). The dynamics highlight the interaction between the price and supply signals.}
        \label{subfig:case3_w}
    \end{subfigure}
    \hfill
    \begin{subfigure}[t]{0.48\textwidth}
        \centering
        \includegraphics[width=\textwidth]{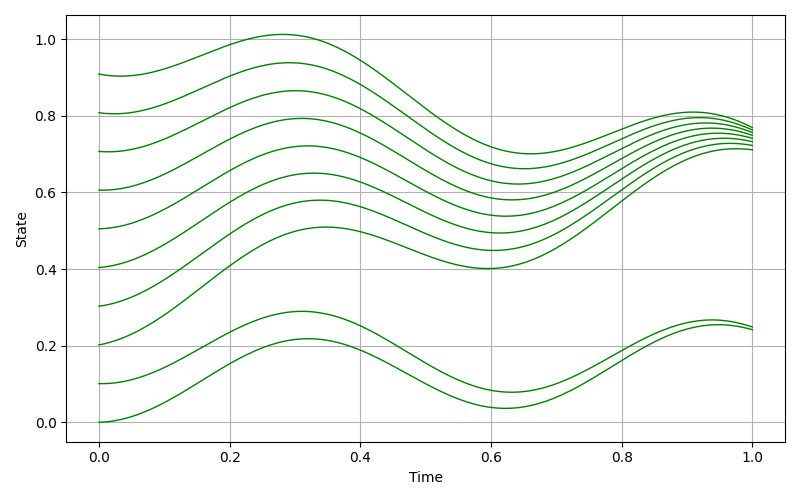}
        \caption{Evolution of state trajectories under the computed equilibrium price, showing consistent convergence patterns across initial conditions.}
        \label{subfig:case3_z}
    \end{subfigure}
    \caption{Equilibrium price and trajectories for the spatial potential $V(z)=0$, terminal cost function $g(z)=25(z-0.25)^2(z-0.75)^2$, and supply function $Q(t)=\sin 10t$.}\label{fig:case3}
\end{figure}

\subsection{Case Four: \(r_3 \neq 0\), \(r_4 = 0\)}

In this case, we activate the running double-well potential while letting the terminal cost to be $0$:
\begin{equation*}
    V(z)=25(z-y_3)^2(z-y_4)^2=25(z-0.25)^2(z-0.75)^2,\quad g(z)=0,
\end{equation*}
which corresponds to choosing $r_3=50,r_4=0$ in~\eqref{eq:g_V_double_well}. We test our method using two distinct types of $Q(t)$: a smooth sinusoidal function $Q(t) = \sin 10t$, and a realization of a Wiener process $Q(t) = W_t(o)$, as performed in Section~\ref{subsec:experiments_case2}.

\subsubsection{Smooth Supply: $Q(t) = \sin 10t$}\label{subsubsec:experiments_case4-1}

Again, Figure~\ref{subfig:case4-1_w} follows a smooth oscillating path, reflecting the oscillatory behavior of $Q(t) = \sin 10t$. Figure~\ref{subfig:case4-1_z} shows the trajectories of $10$ agents. Unlike the example in Section~\ref{subsec:experiments_case3}, the double-well structure influences the agents continuously over time. Hence, the agents cluster and approach the well centers $y_3=0.25$ and $y_4=0.75$ earlier.

\begin{figure}[h]
    \centering
    \begin{subfigure}[t]{0.48\textwidth}
        \centering
        \includegraphics[width=\textwidth]{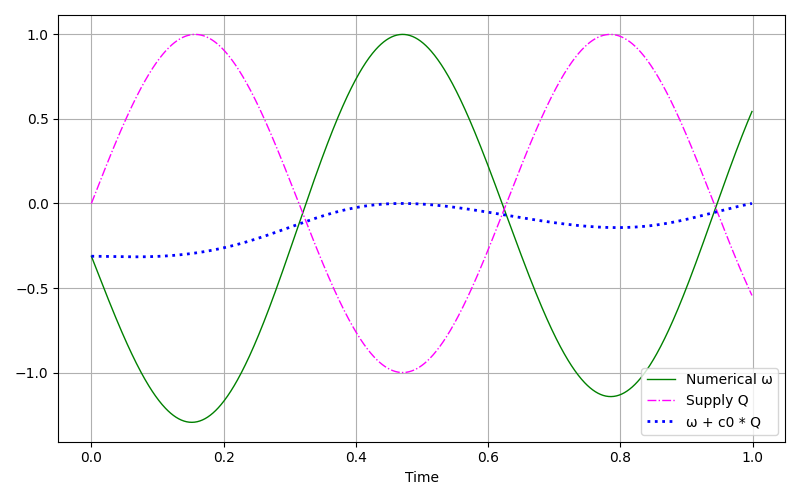}
        \caption{Numerical price trajectory $\omega(t)$ (green, solid) alongside exogenous supply $Q$ (magenta, dashed and dotted) and the combination curve $\omega+c_0 Q$ (dotted, blue). The dynamics highlight the interaction between the price and supply signals.}
        \label{subfig:case4-1_w}
    \end{subfigure}
    \hfill
    \begin{subfigure}[t]{0.48\textwidth}
        \centering
        \includegraphics[width=\textwidth]{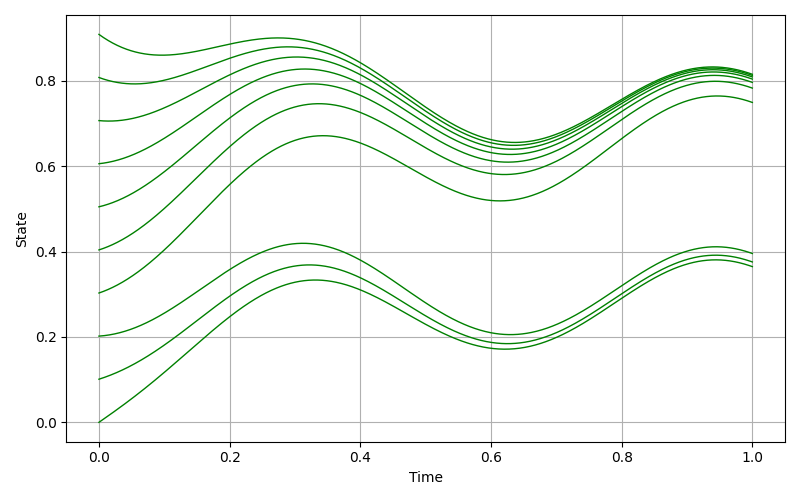}
        \caption{Evolution of state trajectories under the computed equilibrium price, showing earlier convergence patterns due to constantly active double-well spatial potential.}
        \label{subfig:case4-1_z}
    \end{subfigure}
    \caption{Equilibrium price and trajectories for the spatial potential $V(z)=25(z-0.25)^2(z-0.75)^2$, terminal cost function $g(z)=0$, and supply function $Q(t)=\sin 10t$.}\label{fig:case4-1}
\end{figure}

\subsubsection{Non-smooth Supply: $Q(t) = W_t(o)$}

As in Section~\ref{subsec:experiments_case2}, Figure~\ref{subfig:case4-2_w} shows that the optimal price $\omega(t)$ mirrors the non-smoothness of the supply function, and the combination function $\omega+c_0Q $ appears to be smooth. Interestingly, Figure~\ref{subfig:case4-2_z} reveals that the agents do not bifurcate into two clusters as seen in Section~\ref{subsubsec:experiments_case4-1}. Instead, their paths collapse toward a single, narrow region, despite the active double-well running cost. This phenomenon is explained by the fact that, in this case, the supply function has an overall downward trend, and all agents find more beneficial to move toward the lower well centered at $y_3=0.25$.

\begin{figure}[h]
    \centering
    \begin{subfigure}[t]{0.48\textwidth}
        \centering
        \includegraphics[width=\textwidth]{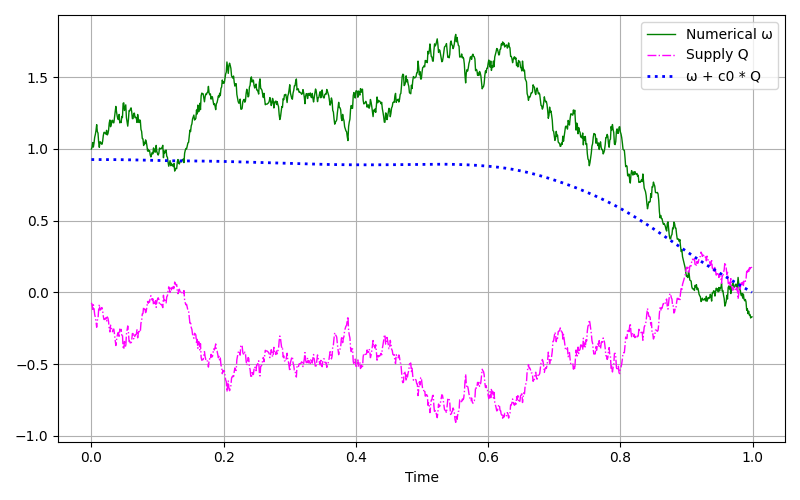}
        \caption{Numerical price trajectory $\omega(t)$ (green, solid) alongside non-smooth exogenous supply $Q$ (magenta, dashed and dotted) and the \textit{smooth} combination curve $\omega+c_0 Q$ (dotted, blue). The dynamics highlight the interaction between the price and supply signals.}
        \label{subfig:case4-2_w}
    \end{subfigure}
    \hfill
    \begin{subfigure}[t]{0.48\textwidth}
        \centering
        \includegraphics[width=\textwidth]{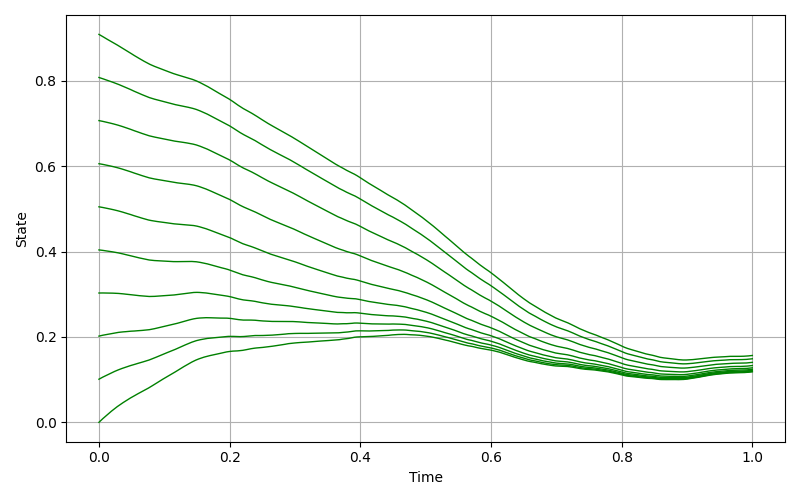}
        \caption{Evolution of state trajectories under the computed equilibrium price, showing a unimodal convergence pattern due to the overall downward trend of the supply function.}
        \label{subfig:case4-2_z}
    \end{subfigure}
    \caption{Equilibrium price and trajectories for the spatial potential $V(z)=25(z-0.25)^2(z-0.75)^2$, terminal cost $g(z)=0$, and supply function $Q(t)=W_t(o)$.}\label{fig:case4-2}
\end{figure}

\section{Conclusion}\label{sec:conclusion}

In this work, we develop a simple yet efficient primal-dual algorithm for computing equilibrium prices in mean-field game price formation models. Our main theoretical contribution establishes that equilibrium prices are optimal solutions to a suitable variational problem, leading to the derivation of Algorithm~\ref{alg:main} with strong convergence properties. In particular, we show the existence of equilibrium prices for all supply functions with sufficient regularity.

Our computational framework leverages modern automatic differentiation techniques to implement Algorithm~\ref{alg:main} in a modular and flexible manner, eliminating the need to manually derive adjoint equations while maintaining computational efficiency. While our current framework successfully leverages automatic differentiation for gradient computations, several promising directions emerge for integrating advanced deep learning techniques with our primal-dual methodology, including generating feedback policies~\cite{ruthotto2019machine, lin2020apacnet} as well as kernel decoupling techniques for models with non-local interactions~\cite{liu2021computational, agrawal2022random, chow2022numerical}.

Future work will address multiple assets, state and control constraints, stochasticity, and other real-world features. 
Since our framework is built on the analysis of the single-agent control problem~\eqref{eq:oc_single_general}, these extensions naturally appear at that level:  
\begin{itemize}
    \item models with $d$ assets yield $d$-dimensional trajectories, $z(t)\in\mathbb{R}^d$,  
    \item models with state and control constraints yield constrained problems, $(z(t),\alpha(t))\in\Omega\times U$,  
    \item stochastic models lead to stochastic trajectories, with a value function
    \[
        \phi_\omega(x)=\inf \, \mathbb{E}_{z(0)=x}\!\left[\int_0^T \big(L(z(t),\alpha(t))+\alpha(t)\omega(t)\big)dt + g(z(T))\right].
    \]
\end{itemize}

The concavity of $\omega\mapsto J[\omega]$ persists (as $\phi_\omega$ remains an infimum of affine functionals), while differentiability must be revisited in the presence of non-unique or stochastic trajectories. To handle this, we expect techniques from perturbation analysis~\cite{bonnansshapiro00perturbation}, measure-theoretic formulations of trajectory dynamics~\cite{bonnans2023lagrangian}, and regularity results for value functions~\cite{cannarsa2018bfc11smoothnessconstrainedsolutionscalculus} to be instrumental.

\section*{Acknowledgments}

We are grateful to the anonymous referees for their careful reading of the manuscript and their insightful suggestions.


\begin{thebibliography}{10}

\bibitem{agrawal2022random}
Sudhanshu Agrawal, Wonjun Lee, Samy~Wu Fung, and Levon Nurbekyan.
\newblock Random features for high-dimensional nonlocal mean-field games.
\newblock {\em Journal of Computational Physics}, 459:111136, 2022.

\bibitem{alharbi_et_al19price}
Abdulrahman Alharbi, Tigran Bakaryan, Rafael Cabral, Sara Campi, Nicholas
  Christoffersen, Paolo Colusso, Odylo Costa, Serikbolsyn Duisembay, Rita
  Ferreira, {Diogo A.} Gomes, Shibei Guo, Julian Gutierrezpineda, Phebe Havor,
  Michele Mascherpa, Simone Portaro, {Ricardo de Lima} Ribeiro, Fernando
  Rodriguez, Johan Ruiz, Fatimah Saleh, Calum Strange, Teruo Tada, Xianjin
  Yang, and Zofia Wr{\'o}blewska.
\newblock A price model with finitely many agents.
\newblock {\em Bulletin of the Portuguese Mathematical Society}, December 2019.
\newblock KAUST Repository Item: Exported on 2020-10-01 Acknowledgements:
  Applied Mathematics Summer School.

\bibitem{ashrafyan22potential}
Yuri Ashrafyan, Tigran Bakaryan, Diogo Gomes, and Julian Gutierrez.
\newblock The potential method for price-formation models.
\newblock In {\em 2022 IEEE 61st Conference on Decision and Control (CDC)},
  pages 7565--7570, 2022.

\bibitem{ashrafyan2021duality}
Yuri Ashrafyan, Tigran Bakaryan, Diogo Gomes, and Julian Gutierrez.
\newblock A duality approach to a price formation {MFG} model.
\newblock {\em Minimax Theory Appl.}, 8(1):1--36, 2023.

\bibitem{ashrafyan2022variational}
Yuri Ashrafyan, Tigran Bakaryan, Diogo Gomes, and Julian Gutierrez.
\newblock A variational approach for price formation models in one dimension.
\newblock {\em Commun. Math. Sci.}, 22(1):227--255, 2024.

\bibitem{ashrafyan2024fullydiscrete}
Yuri Ashrafyan and Diogo Gomes.
\newblock A fully-discrete semi-lagrangian scheme for a price formation mfg
  model.
\newblock {\em Dynamic Games and Applications}, 15(2):503--533, 2025.

\bibitem{aubin90set}
Jean-Pierre Aubin and H\'{e}l\`ene Frankowska.
\newblock {\em Set-valued analysis}, volume~2 of {\em Systems \& Control:
  Foundations \& Applications}.
\newblock Birkh\"{a}user Boston, Inc., Boston, MA, 1990.

\bibitem{betts98survey}
John~T. Betts.
\newblock Survey of numerical methods for trajectory optimization.
\newblock {\em Journal of Guidance, Control, and Dynamics}, 21(2):193--207,
  1998.

\bibitem{bonnans2023lagrangian}
J.~Fr\'{e}d\'{e}ric Bonnans, Justina Gianatti, and Laurent Pfeiffer.
\newblock A lagrangian approach for aggregative mean field games of controls
  with mixed and final constraints.
\newblock {\em SIAM Journal on Control and Optimization}, 61(1):105--134, 2023.

\bibitem{bonnans2021schauder}
J.~Fr{\'e}d{\'e}ric Bonnans, Saeed Hadikhanloo, and Laurent Pfeiffer.
\newblock Schauder estimates for a class of potential mean field games of
  controls.
\newblock {\em Applied Mathematics {\&} Optimization}, 83(3):1431--1464, Jun
  2021.

\bibitem{bonnans2023discrete}
J.~Fr{\'e}d{\'e}ric Bonnans, Pierre Lavigne, and Laurent Pfeiffer.
\newblock Discrete potential mean field games: duality and numerical
  resolution.
\newblock {\em Mathematical Programming}, 202(1):241--278, Nov 2023.

\bibitem{bonnansshapiro00perturbation}
J.~Fr\'ed\'eric Bonnans and Alexander Shapiro.
\newblock {\em Perturbation analysis of optimization problems}.
\newblock Springer Series in Operations Research. Springer-Verlag, New York,
  2000.

\bibitem{cannarsa2018bfc11smoothnessconstrainedsolutionscalculus}
Piermarco Cannarsa, Rossana Capuani, and Pierre Cardaliaguet.
\newblock $\bf{C^{1,1}}$-smoothness of constrained solutions in the calculus of
  variations with application to mean field games.
\newblock 2018.

\bibitem{champock11}
Antonin Chambolle and Thomas Pock.
\newblock A first-order primal-dual algorithm for convex problems with
  applications to imaging.
\newblock {\em J. Math. Imaging Vision}, 40(1):120--145, 2011.

\bibitem{champock16}
Antonin Chambolle and Thomas Pock.
\newblock On the ergodic convergence rates of a first-order primal-dual
  algorithm.
\newblock {\em Math. Program.}, 159(1-2, Ser. A):253--287, 2016.

\bibitem{chow2022numerical}
Yat~Tin Chow, Samy~Wu Fung, Siting Liu, Levon Nurbekyan, and Stanley Osher.
\newblock A numerical algorithm for inverse problem from partial boundary
  measurement arising from mean field game problem.
\newblock {\em Inverse Problems}, 39(1):014001, 2022.

\bibitem{paola19}
A.~{De Paola}, V.~{Trovato}, D.~{Angeli}, and G.~{Strbac}.
\newblock A mean field game approach for distributed control of thermostatic
  loads acting in simultaneous energy-frequency response markets.
\newblock {\em IEEE Transactions on Smart Grid}, 10(6):5987--5999, Nov 2019.

\bibitem{gomes2023machine0}
Diogo Gomes, Julian Gutierrez, and Mathieu Lauri{\`e}re.
\newblock Machine learning architectures for price formation models.
\newblock {\em Applied Mathematics {\&} Optimization}, 88(1):23, May 2023.

\bibitem{gomes23machine}
Diogo Gomes, Julian Gutierrez, and Mathieu Laurière.
\newblock Machine learning architectures for price formation models with common
  noise.
\newblock In {\em 2023 62nd IEEE Conference on Decision and Control (CDC)},
  pages 4345--4350, 2023.

\bibitem{gomes23random}
Diogo Gomes, Julian Gutierrez, and Ricardo Ribeiro.
\newblock A random-supply mean field game price model.
\newblock {\em SIAM Journal on Financial Mathematics}, 14(1):188--222, 2023.

\bibitem{gomes2016acc}
Diogo Gomes, Laurent Lafleche, and Levon Nurbekyan.
\newblock A mean-field game economic growth model.
\newblock In {\em 2016 American Control Conference (ACC)}, pages 4693--4698,
  2016.

\bibitem{nurbekyan_oc_hilbert}
Diogo Gomes and Levon Nurbekyan.
\newblock On the minimizers of calculus of variations problems in {H}ilbert
  spaces.
\newblock {\em Calc. Var. Partial Differential Equations}, 52(1-2):65--93,
  2015.

\bibitem{gomessaude'18}
Diogo~A. Gomes and Jo{\~a}o Sa{\'u}de.
\newblock A mean-field game approach to price formation.
\newblock {\em Dynamic Games and Applications}, 11(1):29--53, Mar 2021.

\bibitem{graber2021weak}
P.~Jameson Graber, Alan Mullenix, and Laurent Pfeiffer.
\newblock Weak solutions for potential mean field games of controls.
\newblock {\em Nonlinear Differential Equations and Applications NoDEA},
  28(5):50, Jul 2021.

\bibitem{HCM07}
M.~Huang, P.~E. Caines, and R.~P. Malham\'e.
\newblock Large-population cost-coupled {LQG} problems with nonuniform agents:
  individual-mass behavior and decentralized {$\epsilon$}-{N}ash equilibria.
\newblock {\em IEEE Trans. Automat. Control}, 52(9):1560--1571, 2007.

\bibitem{HCM06}
M.~Huang, R.~P. Malham\'e, and P.~E. Caines.
\newblock Large population stochastic dynamic games: closed-loop
  {M}c{K}ean-{V}lasov systems and the {N}ash certainty equivalence principle.
\newblock {\em Commun. Inf. Syst.}, 6(3):221--251, 2006.

\bibitem{LasryLions06a}
Jean-Michel Lasry and Pierre-Louis Lions.
\newblock Jeux \`a champ moyen. {I}. {L}e cas stationnaire.
\newblock {\em C. R. Math. Acad. Sci. Paris}, 343(9):619--625, 2006.

\bibitem{LasryLions06b}
Jean-Michel Lasry and Pierre-Louis Lions.
\newblock Jeux \`a champ moyen. {II}. {H}orizon fini et contr\^{o}le optimal.
\newblock {\em C. R. Math. Acad. Sci. Paris}, 343(10):679--684, 2006.

\bibitem{LasryLions2007}
Jean-Michel Lasry and Pierre-Louis Lions.
\newblock Mean field games.
\newblock {\em Jpn. J. Math.}, 2(1):229--260, 2007.

\bibitem{lavigne2023generalized}
Pierre Lavigne and Laurent Pfeiffer.
\newblock Generalized conditional gradient and learning in potential mean field
  games.
\newblock {\em Applied Mathematics {\&} Optimization}, 88(3):89, Oct 2023.

\bibitem{lin2020apacnet}
Alex~Tong Lin, Samy~Wu Fung, Wuchen Li, Levon Nurbekyan, and Stanley~J. Osher.
\newblock Alternating the population and control neural networks to solve
  high-dimensional stochastic mean-field games.
\newblock {\em Proc. Natl. Acad. Sci. USA}, 118(31):Paper No. e2024713118, 10,
  2021.

\bibitem{liu2021computational}
Siting Liu, Matthew Jacobs, Wuchen Li, Levon Nurbekyan, and Stanley~J Osher.
\newblock Computational methods for first-order nonlocal mean field games with
  applications.
\newblock {\em SIAM Journal on Numerical Analysis}, 59(5):2639--2668, 2021.

\bibitem{nursaude18}
Levon Nurbekyan and J.~Sa\'{u}de.
\newblock Fourier approximation methods for first-order nonlocal mean-field
  games.
\newblock {\em Port. Math.}, 75(3-4):367--396, 2018.

\bibitem{parikhboyd14}
Neal Parikh and Stephen Boyd.
\newblock Proximal algorithms.
\newblock {\em Found. Trends Optim.}, 1(3):127–239, January 2014.

\bibitem{paszke2017automatic}
Adam Paszke, Sam Gross, Soumith Chintala, Gregory Chanan, Edward Yang, Zachary
  DeVito, Zeming Lin, Alban Desmaison, Luca Antiga, and Adam Lerer.
\newblock Automatic differentiation in pytorch.
\newblock In {\em NIPS 2017 Workshop on Autodiff}, 2017.

\bibitem{ruthotto2019machine}
Lars Ruthotto, Stanley~J. Osher, Wuchen Li, Levon Nurbekyan, and Samy~Wu Fung.
\newblock A machine learning framework for solving high-dimensional mean field
  game and mean field control problems.
\newblock {\em Proceedings of the National Academy of Sciences},
  117(17):9183--9193, 2020.

\end{thebibliography}

\appendix

\section{Analytic solutions}\label{apdx:analytic}

Here, we derive the analytic expression for the equilibrium price and corresponding optimal trajectories when $V,g$ are given by~\eqref{eq:g_V_quad}. From Proposition~\ref{prp:min_uniqueness} we have that optimal trajectories satisfy the following system of ordinary differential equations
\begin{equation}\label{eq:z_ode}
    \begin{cases}
            c_0\ddot{z}(t,x)+\dot{\omega}(t)=r_1(z(t,x)-y_1),~t\in (0,T),\\
            z(0,x)= x,~c_0\dot{z}(T,x)+\omega(T)=-r_2(z(T,x)-y_2).
        \end{cases}
\end{equation}
Assuming that $\omega$ is an equilibrium price and introducing
\begin{equation*}
    \Lambda(t)=\int_{\mathbb{R}} x \rho_0(x) dx+\int_0^t Q(s) ds,\quad t\in [0,T],
\end{equation*}
we obtain that
\begin{equation*}
    \dot{\Lambda}(t)=Q(t)=\int_{\mathbb{R}}\dot{z}(t,x) \rho_0(x) dx=\frac{d}{dt} \int_{\mathbb{R}}z(t,x) \rho_0(x) dx,\quad \forall t\in (0,T),
\end{equation*}
and
\begin{equation*}
    \Lambda(0)=\int_{\mathbb{R}} x \rho_0(x) dx= \int_{\mathbb{R}} z(0,x) \rho_0(x) dx.
\end{equation*}
Thus, we have that
\begin{equation*}
    \Lambda(t)= \int_{\mathbb{R}} z(t,x) \rho_0(x) dx,\quad \forall t\in [0,T],
\end{equation*}
and integrating~\eqref{eq:z_ode} with respect to $\rho_0$, we obtain
\begin{equation}\label{eq:omega_ODE}
\begin{cases}
    \dot{\omega}(t)=-c_0\ddot{\Lambda}(t)+r_1(\Lambda(t)-y_1),~t\in (0,T),\\
    \omega(T)=-c_0\dot{\Lambda}(T)-r_2(\Lambda(T)-y_2).
\end{cases}
\end{equation}
The latter can be easily integrated:
\begin{equation*}
    \omega(t)=-c_0\dot{\Lambda}(T)-r_2(\Lambda(T)-y_2)-\int_t^T \left[-c_0\ddot{\Lambda}(s)+r_1(\Lambda(s)-y_1)\right] ds.
\end{equation*}

Once $\omega$ is known,~\eqref{eq:z_ode} can also be integrated. Indeed, the general solution of the second-order ordinary differential equation~\eqref{eq:z_ode} is given by
\begin{equation*}
    z(t,x)=y_1+z_{\text{hom}}(t,x)-\frac{1}{c_0}\int_0^t \omega(s)\cosh k(t-s) ds,
\end{equation*}
where $k=\sqrt{\frac{r_1}{c_0}}$, and $z_{\text{hom}}$ is a solution to the homogeneous equation
\begin{equation*}
    \ddot{z}-k^2 z  =0.
\end{equation*}
Furthermore, we have that
\begin{equation}\label{eq:z_gen_sol}
    \dot{z}(t,x)=\dot{z}_{\text{hom}}(t,x)-\frac{\omega(t)}{c_0}-\frac{k}{c_0}\int_0^t \omega(s)\sinh k(t-s) ds,
\end{equation}
so initial-terminal conditions in~\eqref{eq:z_ode} reduce to
\begin{equation}\label{eq:z_boundary}
\begin{cases}
    z_{\text{hom}}(0,x)=x-y_1,\\
    c_0\dot{z}_{\text{hom}}(T,x)+r_2 z_{\text{hom}}(T,x)=B,
\end{cases}
\end{equation}
where, for simplicity, we denote by
\begin{equation*}
    B=r_2(y_2-y_1)+\int_0^T \omega(s) \left[\frac{r_2}{c_0}\cosh k(T-s)+k \sinh k(T-s)\right] ds.
\end{equation*}
There are two possibilities.
\begin{enumerate}
    \item $r_1=0$. In this case, we have that $k=0$, and
    \begin{equation*}
        \ddot{z}_{\text{hom}}(t,x)=0 \Longrightarrow z_{\text{hom}}(t,x)=c_1(x)+c_2(x)t.
    \end{equation*}
    Furthermore,~\eqref{eq:z_boundary} yields
    \begin{equation*}
        z_{\text{hom}}(t,x)=x-y_1+\frac{B-r_2(x-y_1)}{c_0+r_2 T}t,
    \end{equation*}
    and
    \begin{equation*}
        z(t,x)=x+\frac{B-r_2(x-y_1)}{c_0+r_2 T}t-\frac{1}{c_0}\int_0^t \omega(s) ds.
    \end{equation*}
    \item $r_1\neq 0$. In this case, we have that $k \neq 0$, and
    \begin{equation*}
        z_{\text{hom}}(t,x)=c_3(x) \cosh kt+c_4(x) \sinh kt.
    \end{equation*}
    Furthermore,~\eqref{eq:z_boundary} yields
    \begin{equation*}
        z_{\text{hom}}(t,x)=(x-y_1) \cosh kt+  \frac{B-(x-y_1)(c_0 k \sinh kT+r_2 \cosh kT)}{c_0 k \cosh kT+r_2 \sinh kT} \sinh kt,
    \end{equation*}
    and
    \begin{equation*}
    \begin{split}
        z(t,x)=&y_1+(x-y_1) \cosh kt+ \frac{B-(x-y_1)(c_0 k \sinh kT+r_2 \cosh kT)}{c_0 k \cosh kT+r_2 \sinh kT} \sinh k t\\
        &-\frac{1}{c_0} \int_0^t \omega(s) \cosh k(t-s) d s.
    \end{split}
    \end{equation*}
\end{enumerate}

\end{document}